\documentclass[a4paper]{article}

\usepackage{geometry}
\usepackage{amsmath,amssymb,amsthm}
\usepackage{latexsym}
\usepackage{indentfirst}
\usepackage{graphicx}
\usepackage{mathrsfs}
\usepackage{chngcntr}
\usepackage{color}
\usepackage[colorlinks=true]{hyperref}

\numberwithin{equation}{section}

\newtheorem{theorem}{Theorem}[section]
\newtheorem{Cor}[theorem]{Corollary}
\newtheorem{lemma}[theorem]{Lemma}
\newtheorem{Pro}[theorem]{Proposition}

\newtheorem{Def}[theorem]{Definition}

\newcommand{\N}{\mathcal{N}}

\newcommand{\R}{\mathbb{R}}

\newcommand{\D}{\mathcal{D}}

\title{Existence and Concentration of Multiple Positive Solutions for a Logarithmic Fractional Schr\"odinger--Poisson System}

\author{Jiao Luo, Zhipeng Yang\thanks{Corresponding author: yangzhipeng326@163.com}}

\date{}

\AtEndDocument{%
  \par
  \bigskip
  \bigskip

  \noindent
  \textbf{Jiao Luo}\\[0.2em]
  \textsc{Department of Mathematics, Yunnan Normal University, Kunming, China}\\[0.3em]
  \textit{E-mail address}: \texttt{luojiao20230401@163.com}\\[1.5em]

  \noindent
  \textbf{Zhipeng Yang}\\[0.2em]
  \textsc{Department of Mathematics, Yunnan Normal University, Kunming, China}\\
  \textsc{Yunnan Key Laboratory of Modern Analytical Mathematics and Applications, Kunming, China}\\[0.3em]
  \textit{E-mail address}: \texttt{yangzhipeng326@163.com}%
}

\begin{document}

\maketitle

\begin{abstract}
We study a logarithmic fractional Schr\"odinger--Poisson system in \(\R^{3}\):
\begin{equation*}
\begin{cases}
\varepsilon^{2\alpha}(-\Delta)^{\alpha}u+V(x)u+\phi u=u\log u^{2}+|u|^{p-2}u, & \text{in }\R^{3},\\
\varepsilon^{2\alpha}(-\Delta)^{\alpha}\phi=u^{2}, & \text{in }\R^{3}.
\end{cases}
\end{equation*}
Here \(\alpha\in\bigl(\frac34,1\bigr)\), \(4<p<2_{\alpha}^{*}=\frac{6}{3-2\alpha}\), and \(V\) satisfies a global potential condition. Using a suitable Orlicz-type Banach space, we establish a \(C^{1}\) variational framework for the problem and combine the Nehari manifold method with Lusternik--Schnirelmann category theory. We then prove that, for every fixed \(\delta>0\) and all sufficiently small \(\varepsilon>0\), the system admits at least \(\operatorname{cat}_{M_{\delta}}(M)\) distinct positive solutions. Moreover, the maximum points of these solutions concentrate near the global minimum set of \(V\) as \(\varepsilon\to0\).
\end{abstract}

\vspace{3mm}
\noindent\textbf{Keywords:} Fractional Schr\"odinger--Poisson system; Lusternik--Schnirelmann category; logarithmic term.

\vspace{3mm}
\noindent\textbf{MSC2020 Mathematics Subject Classification:} 35J20; 35A01; 58E05.

\tableofcontents

\section{Introduction}

In this paper, we consider the logarithmic fractional Schr\"odinger--Poisson system
\begin{equation}\label{eq1.1}
\begin{cases}
\varepsilon^{2\alpha}(-\Delta)^{\alpha}u+V(x)u+\phi u=u\log u^{2}+|u|^{p-2}u, & \text{in }\R^{3},\\
\varepsilon^{2\alpha}(-\Delta)^{\alpha}\phi=u^{2}, & \text{in }\R^{3}.
\end{cases}
\end{equation}
Here \(\varepsilon>0\) is a small parameter, and \((-\Delta)^{\alpha}\) denotes the fractional Laplacian of order \(\alpha\in(0,1)\). Throughout the paper we assume
\[
\alpha\in\Bigl(\frac34,1\Bigr),
\qquad
4<p<2_{\alpha}^{*}:=\frac{6}{3-2\alpha},
\]
and
\[
(V)\qquad
V\in C^{1}(\R^{3},\R),
\qquad
V_{\infty}:=\lim_{|x|\to\infty}V(x)>V_{0}:=\inf_{x\in\R^{3}}V(x)>-1.
\]
Condition \((V)\) is a standard global potential assumption in semiclassical analysis and, in particular, guarantees that the minimum set of \(V\) is nonempty and compact.

System \eqref{eq1.1} combines three distinct features. The first is the fractional kinetic operator \((-\Delta)^\alpha\), which arises in models of anomalous diffusion and in fractional Schr\"odinger equations; see \cite{MR3002595,Chang_2013,MR2944369}. The second is the Poisson coupling, which gives rise to the additional nonlocal interaction term \(\phi u\); see, for example, \cite{Teng2016,MR3456743,MR3716931,MR4230968}. The third is the logarithmic nonlinearity, which appears in quantum mechanics, quantum optics, effective quantum gravity, transport theory, and Bose--Einstein condensation \cite{MR583902,MR719365,MR2740900}. In particular, the term
\[
u\log u^{2}
\]
leads to a variational structure markedly different from that associated with pure power nonlinearities. We also mention the sharp Euclidean \(L^{p}\)-Sobolev logarithmic inequality of Del Pino and Dolbeault \cite{MR1957678}, which highlights the role of logarithmic corrections in borderline interpolation phenomena.

Solving formally the second equation in \eqref{eq1.1}, one obtains
\[
\phi(x)=\varepsilon^{-2\alpha}(I_{2\alpha}*u^{2})(x),
\qquad
I_{2\alpha}(x)=\frac{A_{\alpha}}{|x|^{3-2\alpha}},
\]
so that \eqref{eq1.1} may be viewed as a fractional Schr\"odinger equation involving both a logarithmic term and a Hartree-type nonlocal interaction. After the standard semiclassical change of variables
\[
x\mapsto \varepsilon x,
\]
and still denoting the rescaled functions by \(u\) and \(\phi\), we arrive at
\begin{equation}\label{eq1.2}
\begin{cases}
(-\Delta)^\alpha u+V(\varepsilon x)u+\phi u=u\log u^{2}+|u|^{p-2}u, & \text{in }\R^{3},\\
(-\Delta)^\alpha\phi=u^{2}, & \text{in }\R^{3}.
\end{cases}
\end{equation}
The concentration behavior is therefore governed by the coefficient \(V(\varepsilon x)\) as \(\varepsilon\to0\).

The logarithmic Schr\"odinger equation has been extensively studied. Early results go back to Cazenave--Haraux \cite{MR583902} and Cazenave \cite{MR719365}. In the variational framework, d'Avenia, Montefusco and Squassina \cite{MR3195154}, Squassina and Szulkin \cite{MR3385171}, and Ji and Szulkin \cite{MR3451965} developed effective methods for logarithmic Schr\"odinger equations under periodic or asymptotically periodic assumptions. In the semiclassical setting, Alves and de Morais Filho \cite{MR3869846} proved the existence and concentration of positive solutions for
\[
-\varepsilon^{2}\Delta u+V(x)u=u\log u^{2}
\]
under a global potential condition. Later, Alves and Ji \cite{MR4048330} established concentration results via penalization, and in \cite{MR4097474} they obtained multiplicity of positive solutions. More recently, Alves and Ji \cite{MR4732981} constructed multi-peak positive solutions for a logarithmic Schr\"odinger equation by variational methods. In the fractional setting, Ji and Xue \cite{MR4466439} studied existence and concentration for a fractional logarithmic Schr\"odinger equation, while Alves and Ambrosio \cite{MR4698809} considered a fractional \(p\)-Laplacian logarithmic Schr\"odinger equation.

Fractional Schr\"odinger--Poisson equations have been widely studied in recent years. The existing literature includes results on ground states, multiplicity and concentration of positive solutions, normalized solutions, critical growth problems, discontinuous nonlinearities, and penalization methods; see, for instance, \cite{Teng2016,MR4870288,MR3456743,MR3716931,MR3994307,MR3996970,MR4072024,MR4427207,MR4559061,MR4803716,MR4230968}. In particular, under global assumptions on the potential, several authors proved multiplicity and concentration results for semiclassical solutions.
For logarithmic fractional Schr\"odinger--Poisson systems, however, much less is known. Tao and Li \cite{MR4782770} proved the existence of a positive solution in the saddle-like case, while Li and Tao \cite{MR4896434} established multi-bump solutions in the deepening potential well case. Related results based on penalization methods were further obtained in \cite{luo2025existence}. Nevertheless, as far as we know, Lusternik--Schnirelmann multiplicity for logarithmic fractional Schr\"odinger--Poisson systems under the global potential condition \((V)\) has not yet been considered.

A central question in semiclassical analysis is how the number of solutions reflects the topology of the set
\[
M=\{x\in\R^{3}:V(x)=V_{0}\},
\qquad
M_{\delta}=\{x\in\R^{3}:\operatorname{dist}(x,M)\le\delta\},
\]
where \(\delta>0\) is fixed. Under \((V)\), the set \(M\) is nonempty and compact. This perspective goes back to the local case and is closely related to the works of del Pino and Felmer \cite{MR1379196}, Cingolani and Lazzo \cite{MR1646619}, and Rabinowitz \cite{MR1162728}. In the present setting, several additional difficulties arise.

The first difficulty comes from the logarithmic term. Although the primitive
\[
F(s)=\frac12 s^{2}\log s^{2}-\frac12 s^{2},
\qquad s\in\R,
\]
extends to a \(C^{1}\) function on \(\R\), the functional
\[
u\mapsto \int_{\R^{3}}F(u)\,dx
\]
is not well defined on the whole space \(H^{\alpha}(\R^{3})\); indeed, it may take the value \(-\infty\). Thus the natural energy functional is not of class \(C^{1}\) on \(H^{\alpha}(\R^{3})\). Following \cite{MR3385171,MR3451965,MR3869846}, we decompose
\[
\frac12 s^{2}\log s^{2}=F_{2}(s)-F_{1}(s),
\]
where \(F_{1}\) is convex and \(F_{2}\in C^{1}(\R)\). This leads to the Orlicz-type space
\[
X_{\varepsilon}=H_{\varepsilon}\cap L^{F_{1}}(\R^{3}),
\]
where
\[
H_{\varepsilon}
=
\left\{
u\in H^{\alpha}(\R^{3}):
\int_{\R^{3}}V(\varepsilon x)u^{2}\,dx<\infty
\right\},
\]
and to the functional
\[
\begin{aligned}
J_{\varepsilon}(u)
&=
\frac12\int_{\R^{3}}
\Bigl(
|(-\Delta)^{\frac{\alpha}{2}}u|^{2}
+(V(\varepsilon x)+1)|u|^{2}
\Bigr)\,dx
+\frac14\int_{\R^{3}}\phi_{u}^{\alpha}u^{2}\,dx \\
&\quad
+\int_{\R^{3}}F_{1}(u)\,dx
-\int_{\R^{3}}F_{2}(u)\,dx
-\frac1p\int_{\R^{3}}|u|^{p}\,dx,
\end{aligned}
\]
which is of class \(C^{1}\) on \(X_{\varepsilon}\).

The second difficulty is due to the Poisson term. If \(\phi_{u}^{\alpha}\) denotes the unique solution of
\[
(-\Delta)^{\alpha}\phi=u^{2}
\qquad \text{in }\R^{3},
\]
then
\[
\phi_{u}^{\alpha}(x)=C_{\alpha}\int_{\R^{3}}\frac{u^{2}(y)}{|x-y|^{3-2\alpha}}\,dy,
\]
and the corresponding contribution to the energy is quartic:
\[
\int_{\R^{3}}\phi_{u}^{\alpha}u^{2}\,dx.
\]
By the Hardy--Littlewood--Sobolev inequality and the embedding
\[
H^{\alpha}(\R^{3})\hookrightarrow L^{\frac{12}{3+2\alpha}}(\R^{3}),
\]
we have
\[
\int_{\R^{3}}\phi_{u}^{\alpha}u^{2}\,dx
\le
C\|u\|_{\frac{12}{3+2\alpha}}^{4}.
\]
This is one reason for assuming
\[
\alpha>\frac34.
\]
The same restriction also enters the decay analysis, since the Bessel kernel associated with \((-\Delta)^{\alpha}+\mu\) belongs to \(L^{2}(\R^{3})\) for \(\alpha>\frac34\).

The third difficulty is the lack of compactness on \(\R^{3}\). Even after restricting the functional to the Nehari manifold
\[
\mathcal N_{\varepsilon}
=
\left\{
u\in X_{\varepsilon}\setminus\{0\}:
\langle J_{\varepsilon}'(u),u\rangle=0
\right\},
\]
one still needs to rule out vanishing, dichotomy, and escape of mass to infinity. A key role is played by the autonomous energy levels
\[
c_{\mu}
=
\inf_{u\in X\setminus\{0\}}\max_{t\ge0}J_{\mu}(tu),
\qquad \mu>-1,
\]
where
\[
X:=H^{\alpha}(\R^{3})\cap L^{F_{1}}(\R^{3}),
\]
and
\[
J_{\mu}(u)
=
\frac12\int_{\R^{3}}
\Bigl(
|(-\Delta)^{\frac{\alpha}{2}}u|^{2}
+(\mu+1)|u|^{2}
\Bigr)\,dx
+\frac14\int_{\R^{3}}\phi_{u}^{\alpha}u^{2}\,dx
+\int_{\R^{3}}F_{1}(u)\,dx
-\int_{\R^{3}}F_{2}(u)\,dx
-\frac1p\int_{\R^{3}}|u|^{p}\,dx.
\]
The strict inequality
\[
c_{V_{0}}<c_{V_{\infty}}
\]
provides the compactness threshold needed in the variational argument.

A further issue is positivity. Since our multiplicity argument is carried out on low-energy levels of \(\left.J_{\varepsilon}\right|_{\mathcal N_{\varepsilon}}\), it is not enough to produce a single positive ground state. One must also show that every low-energy critical point is positive. To this end, we prove that such critical points, after suitable translations, converge strongly to a positive ground state of the autonomous limit problem. We then combine local uniform convergence, uniform \(L^\infty\)-bounds, and uniform decay at infinity to exclude sign changes.

Unlike recent works based on penalization methods such as \cite{MR4230968,luo2025existence}, our approach works directly with a \(C^{1}\) functional on an Orlicz-type space and combines the Nehari manifold method with Lusternik--Schnirelmann category theory. Rather than modifying the nonlinearity outside suitable regions, we exploit the decomposition of the logarithmic term, the geometry of the Nehari manifold, and the compactness threshold
\[
c_{V_{0}}<c_{V_{\infty}}
\]
to obtain multiplicity and concentration of positive solutions from the topology of the global minimum set \(M\).

For simplicity, we carry out the proof for the rescaled problem \eqref{eq1.2}. By the inverse change of variables, the corresponding conclusions immediately yield solutions of the original system \eqref{eq1.1}. Our main result is as follows.

\begin{theorem}\label{thm1.1}
Assume that \(\alpha\in\bigl(\frac34,1\bigr)\), \(4<p<2_\alpha^*\), and \((V)\) holds. Then, for every fixed \(\delta>0\), there exists \(\varepsilon_1=\varepsilon_1(\delta)>0\) such that, for every \(\varepsilon\in(0,\varepsilon_1)\), problem \eqref{eq1.1} has at least \(\operatorname{cat}_{M_\delta}(M)\) distinct positive solutions. Moreover, if \((u_{\varepsilon},\phi_{\varepsilon})\) is one of these positive solutions and \(\eta_{\varepsilon}\in\R^{3}\) is a global maximum point of \(u_{\varepsilon}\), then
\[
\lim_{\varepsilon\to0}V(\eta_{\varepsilon})=V_{0}.
\]
\end{theorem}

The paper is organized as follows. In Section 2, we recall the basic facts on Orlicz spaces and fractional Sobolev spaces and introduce the variational framework for problem \eqref{eq1.2}. In Section 3, we study the Nehari manifold and prove the existence of a ground state solution. In Section 4, we analyze the autonomous problem and derive the concentration tools used later. In Section 5, we prove positivity for low-energy critical points, establish the multiplicity of positive solutions via Lusternik--Schnirelmann theory, and complete the proof of Theorem~\ref{thm1.1}.

\section{Preliminaries and Variational Framework}

In this section, we briefly recall the basic facts on Orlicz spaces and fractional Sobolev spaces that will be used later, and then introduce the variational framework for problem \eqref{eq1.2}; see \cite{MR2424078,MR2271234,MR1113700,MR2944369} for background.

\subsection{Orlicz spaces and fractional Sobolev spaces}

\begin{Def}\label{def2.1}
A continuous function \(\Phi:\R\to[0,+\infty)\) is called an \(N\)-function if
\begin{itemize}
    \item[(i)] \(\Phi\) is convex;
    \item[(ii)] \(\Phi(t)=0\) if and only if \(t=0\);
    \item[(iii)] \(\displaystyle \lim_{t\to0^+}\frac{\Phi(t)}{t}=0\) and \(\displaystyle \lim_{t\to+\infty}\frac{\Phi(t)}{t}=+\infty\);
    \item[(iv)] \(\Phi\) is even.
\end{itemize}
\end{Def}

We say that an \(N\)-function \(\Phi\) satisfies the \(\Delta_2\)-condition, and write \(\Phi\in(\Delta_2)\), if there exists \(k>0\) such that
\[
\Phi(2t)\le k\,\Phi(t)
\qquad \forall\, t\ge0.
\]
Its conjugate \(N\)-function is defined by
\[
\widetilde{\Phi}(s)=\sup_{t\ge0}\{|s|t-\Phi(t)\},
\qquad s\in\R.
\]
Then \(\widetilde{\Phi}\) is also an \(N\)-function, and \(\widetilde{\widetilde{\Phi}}=\Phi\).

Let \(\Omega\subset\R^N\) be an open set. The Orlicz space associated with \(\Phi\) is
\[
L^\Phi(\Omega)
=
\left\{
u\in L^1_{\mathrm{loc}}(\Omega):
\int_\Omega \Phi\!\left(\frac{|u|}{\lambda}\right)\,dx<+\infty
\text{ for some }\lambda>0
\right\},
\]
endowed with the Luxemburg norm
\[
\|u\|_\Phi
=
\inf\left\{
\lambda>0:
\int_\Omega \Phi\!\left(\frac{|u|}{\lambda}\right)\,dx\le1
\right\}.
\]
When \(\Omega=\R^3\), we simply write \(L^\Phi\).

We shall use Young's inequality
\[
|st|\le \Phi(t)+\widetilde{\Phi}(s),
\qquad \forall\, s,t\in\R,
\]
and H\"older's inequality
\[
\left|\int_\Omega uv\,dx\right|
\le 2\|u\|_\Phi\|v\|_{\widetilde{\Phi}},
\qquad
\forall\, u\in L^\Phi(\Omega),\ v\in L^{\widetilde{\Phi}}(\Omega).
\]

If \(\Phi,\widetilde{\Phi}\in(\Delta_2)\), then \(L^\Phi(\Omega)\) is reflexive and separable, and
\[
L^\Phi(\Omega)
=
\left\{
u\in L^1_{\mathrm{loc}}(\Omega):
\int_\Omega \Phi(|u|)\,dx<+\infty
\right\}.
\]
Moreover,
\[
u_n\to u \text{ in }L^\Phi(\Omega)
\quad\Longleftrightarrow\quad
\int_\Omega \Phi(|u_n-u|)\,dx\to0.
\]

We also recall the following standard criterion. Let \(\Phi\) be a \(C^1\) \(N\)-function with conjugate \(\widetilde{\Phi}\). Assume that
\begin{equation}\label{eq2.1}
1<l\le \frac{\Phi'(t)t}{\Phi(t)}\le m<\infty,
\qquad t>0.
\end{equation}
Then \(\Phi,\widetilde{\Phi}\in(\Delta_2)\).

Finally, define
\[
\xi_0(t)=\min\{t^l,t^m\},
\qquad
\xi_1(t)=\max\{t^l,t^m\},
\qquad t\ge0.
\]
Under \eqref{eq2.1}, one has
\begin{equation}\label{eq2.2}
\xi_0(\|u\|_\Phi)
\le
\int_\Omega \Phi(|u|)\,dx
\le
\xi_1(\|u\|_\Phi),
\qquad
\forall\, u\in L^\Phi(\Omega).
\end{equation}

For \(\alpha\in(0,1)\), the fractional Sobolev space \(H^\alpha(\R^3)=W^{\alpha,2}(\R^3)\) is defined by
\[
H^\alpha(\R^3)
=
\left\{
u\in L^2(\R^3):
\int_{\R^3}
\bigl(|\xi|^{2\alpha}|\mathcal F(u)(\xi)|^2+|\mathcal F(u)(\xi)|^2\bigr)\,d\xi
<\infty
\right\},
\]
with norm
\[
\|u\|_{H^\alpha(\R^3)}^2
=
\int_{\R^3}
\bigl(|\xi|^{2\alpha}|\mathcal F(u)(\xi)|^2+|\mathcal F(u)(\xi)|^2\bigr)\,d\xi.
\]
We denote by
\[
[u]_{H^\alpha(\R^3)}
=
\left(
\iint_{\R^3\times\R^3}
\frac{|u(x)-u(y)|^2}{|x-y|^{3+2\alpha}}
\,dx\,dy
\right)^{\frac12}
\]
the Gagliardo seminorm. The homogeneous space \(\D^{\alpha,2}(\R^3)\) is the completion of \(C_0^\infty(\R^3)\) with respect to this seminorm, namely,
\[
\D^{\alpha,2}(\R^3)
=
\overline{C_0^\infty(\R^3)}^{\, [\cdot]_{H^\alpha(\R^3)}},
\]
and we equip it with the norm
\[
\|u\|_{\D^{\alpha,2}(\R^3)}=[u]_{H^\alpha(\R^3)}.
\]

Set
\[
2_\alpha^*=\frac{6}{3-2\alpha}.
\]
Then
\[
\D^{\alpha,2}(\R^3)\hookrightarrow L^{2_\alpha^*}(\R^3)
\]
continuously, and
\[
S_\alpha
=
\inf_{u\in \D^{\alpha,2}(\R^3)\setminus\{0\}}
\frac{\|u\|_{\D^{\alpha,2}(\R^3)}^2}{\|u\|_{2_\alpha^*}^2}
\]
is well defined. Moreover,
\[
H^\alpha(\R^3)\hookrightarrow L^q(\R^3)
\quad\text{continuously for all } q\in[2,2_\alpha^*],
\]
and
\[
H^\alpha(\R^3)\hookrightarrow L^q_{\mathrm{loc}}(\R^3)
\quad\text{compactly for all } q\in[2,2_\alpha^*).
\]

For \(u\in\mathcal S(\R^3)\), the fractional Laplacian is defined by
\[
\mathcal F\bigl((-\Delta)^\alpha u\bigr)(\xi)
=
|\xi|^{2\alpha}\mathcal F(u)(\xi),
\qquad \xi\in\R^3,
\]
where
\[
\mathcal F(\phi)(\xi)
=
\frac{1}{(2\pi)^{\frac32}}
\int_{\R^3}e^{-i\xi\cdot x}\phi(x)\,dx,
\qquad \phi\in\mathcal S(\R^3).
\]
It also admits the representation
\[
(-\Delta)^\alpha u(x)
=
-\frac12 C(\alpha)
\int_{\R^3}
\frac{u(x+y)+u(x-y)-2u(x)}{|y|^{3+2\alpha}}
\,dy,
\qquad x\in\R^3,
\]
where
\[
C(\alpha)
=
\left(
\int_{\R^3}\frac{1-\cos \xi_1}{|\xi|^{3+2\alpha}}\,d\xi
\right)^{-1},
\qquad
\xi=(\xi_1,\xi_2,\xi_3).
\]
By Plancherel's formula, there exists \(c_\alpha>0\) such that
\[
[u]_{H^\alpha(\R^3)}^2
=
c_\alpha\left\|(-\Delta)^{\frac{\alpha}{2}}u\right\|_2^2.
\]
Hence the following norms are equivalent on \(H^\alpha(\R^3)\):
\[
u\mapsto
\left(
\|u\|_2^2+[u]_{H^\alpha(\R^3)}^2
\right)^{\frac12},
\]
\[
u\mapsto
\left(
\int_{\R^3}
\bigl(|\xi|^{2\alpha}|\mathcal F(u)(\xi)|^2+|\mathcal F(u)(\xi)|^2\bigr)\,d\xi
\right)^{\frac12},
\]
and
\[
u\mapsto
\left(
\int_{\R^3}
\left(
|u|^2+\bigl|(-\Delta)^{\frac{\alpha}{2}}u\bigr|^2
\right)\,dx
\right)^{\frac12}.
\]

\subsection{The logarithmic decomposition and the working space}

To avoid confusion with the fixed neighbourhood \(M_\delta\) in Theorem~\ref{thm1.1}, we denote by \(\delta_0\in(0,e^{-3/2})\) the truncation parameter used below.

Consider
\[
F(t)=\int_0^t s\log s^2\,ds
=
\frac12 t^2\log t^2-\frac12 t^2,
\qquad t\in\R.
\]
Following \cite{MR3869846,MR4732981,MR3451965,MR3385171}, define
\[
F_1(s)=
\begin{cases}
0, & s=0,\\[0.4em]
-\dfrac12 s^2\log s^2, & 0<|s|<\delta_0,\\[0.4em]
-\dfrac12 s^2(\log \delta_0^2+3)+2\delta_0|s|-\dfrac{\delta_0^2}{2}, & |s|\ge\delta_0,
\end{cases}
\]
and
\[
F_2(s)=
\begin{cases}
0, & |s|<\delta_0,\\[0.4em]
\dfrac12 s^2\log\!\left(\dfrac{s^2}{\delta_0^2}\right)+2\delta_0|s|-\dfrac32 s^2-\dfrac{\delta_0^2}{2}, & |s|\ge\delta_0,
\end{cases}
\]
for all \(s\in\R\). Then
\[
F_2(s)-F_1(s)=\frac12 s^2\log s^2,
\qquad
F_2'(s)-F_1'(s)=s\log s^2+s,
\qquad \forall\, s\in\R,
\]
where
\[
F_1'(s)=
\begin{cases}
0, & s=0,\\[0.4em]
-(\log s^2+1)s, & 0<|s|<\delta_0,\\[0.4em]
-(\log \delta_0^2+3)s+2\delta_0\operatorname{sgn}(s), & |s|\ge\delta_0,
\end{cases}
\]
and
\[
F_2'(s)=
\begin{cases}
0, & |s|<\delta_0,\\[0.4em]
s\log\!\left(\dfrac{s^2}{\delta_0^2}\right)-2s+2\delta_0\operatorname{sgn}(s), & |s|\ge\delta_0.
\end{cases}
\]

The functions \(F_1\) and \(F_2\) satisfy the following standard properties:
\begin{itemize}
    \item[\((f_1)\)] \(F_1\) is even, \(F_1\ge0\), \(F_1'(s)s\ge0\) for all \(s\in\R\), and \(F_1\in C^1(\R,\R)\) is convex;
    \item[\((f_2)\)] \(F_2\in C^1(\R,\R)\cap C^2((-\infty,-\delta_0)\cup(\delta_0,+\infty),\R)\), and for every \(r\in(2,2_\alpha^*)\) there exists \(C_r>0\) such that
    \[
    |F_2'(s)|\le C_r|s|^{r-1},
    \qquad \forall\, s\in\R;
    \]
    in particular, since \(p\in(4,2_\alpha^*)\), there exists \(C>0\) such that
    \[
    |F_2'(s)|\le C\bigl(|s|+|s|^{p-1}\bigr),
    \qquad \forall\, s\in\R;
    \]
    \item[\((f_3)\)] the map \(s\mapsto \dfrac{F_2'(s)}{s}\) is nondecreasing on \((0,+\infty)\) and strictly increasing on \((\delta_0,+\infty)\);
    \item[\((f_4)\)]
    \[
    \lim_{s\to+\infty}\frac{F_2'(s)}{s}=+\infty.
    \]
\end{itemize}

\begin{Pro}\label{prop2.1}
The function \(F_1\) is an \(N\)-function. Moreover, both \(F_1\) and its conjugate \(\widetilde{F}_1\) satisfy the \(\Delta_2\)-condition.
\end{Pro}

\begin{proof}
By construction, \(F_1\) is continuous, even, nonnegative, and vanishes only at the origin. Since \(\delta_0\in(0,e^{-3/2})\), it is also convex on \(\R\). Moreover,
\[
\lim_{s\to0}\frac{F_1(s)}{|s|}
=
\lim_{s\to0}-\frac12 |s|\log s^2
=
0.
\]
If we set
\[
A_0=-(\log \delta_0^2+3)>0,
\]
then for \(|s|\ge\delta_0\),
\[
F_1(s)=\frac{A_0}{2}s^2+2\delta_0|s|-\frac{\delta_0^2}{2},
\]
and therefore
\[
\lim_{|s|\to+\infty}\frac{F_1(s)}{|s|}=+\infty.
\]
Hence \(F_1\) is an \(N\)-function.

It remains to verify \eqref{eq2.1}. Since \(F_1\) is even, it is enough to consider \(s>0\). For \(0<s<\delta_0\),
\[
\frac{F_1'(s)s}{F_1(s)}
=
\frac{-(\log s^2+1)s^2}{-\frac12 s^2\log s^2}
=
2+\frac{1}{\log s},
\]
so, using \(\delta_0<e^{-1}\),
\[
1<
2+\frac{1}{\log \delta_0}
\le
\frac{F_1'(s)s}{F_1(s)}
<
2.
\]
For \(s\ge\delta_0\),
\[
F_1(s)=\frac{A_0}{2}s^2+2\delta_0 s-\frac{\delta_0^2}{2},
\qquad
F_1'(s)=A_0 s+2\delta_0,
\]
hence
\[
\frac{F_1'(s)s}{F_1(s)}
=
\frac{A_0 s^2+2\delta_0 s}{\frac{A_0}{2}s^2+2\delta_0 s-\frac{\delta_0^2}{2}}.
\]
A direct computation shows that
\[
\frac{F_1'(s)s}{F_1(s)}-1
=
\frac{\frac{A_0}{2}s^2+\frac{\delta_0^2}{2}}{\frac{A_0}{2}s^2+2\delta_0 s-\frac{\delta_0^2}{2}}
>0
\]
and
\[
2-\frac{F_1'(s)s}{F_1(s)}
=
\frac{2\delta_0 s-\delta_0^2}{\frac{A_0}{2}s^2+2\delta_0 s-\frac{\delta_0^2}{2}}
>0.
\]
Thus
\[
1<\frac{F_1'(s)s}{F_1(s)}<2,
\qquad \forall\, s>0.
\]
Therefore \eqref{eq2.1} holds for some \(l\in(1,2)\) and \(m=2\), and the conclusion follows.
\end{proof}

Under the standing assumption \(\alpha\in\bigl(\frac34,1\bigr)\), for every \(u\in H^\alpha(\R^3)\) there exists a unique \(\phi_u^\alpha\in\D^{\alpha,2}(\R^3)\) such that
\[
(-\Delta)^\alpha\phi=u^2
\qquad \text{in }\R^3,
\]
and
\[
\phi_u^\alpha(x)
=
C_\alpha\int_{\R^3}\frac{u^2(y)}{|x-y|^{3-2\alpha}}\,dy,
\qquad x\in\R^3,
\]
where
\[
C_\alpha
=
\frac{\Gamma\!\left(\frac{3-2\alpha}{2}\right)}
{\pi^{\frac32}2^{2\alpha}\Gamma(\alpha)}.
\]
Since
\[
F_2'(u)-F_1'(u)=u\log u^2+u,
\]
the rescaled problem \eqref{eq1.2} can be rewritten as
\[
(-\Delta)^\alpha u+(V(\varepsilon x)+1)u+\phi_u^\alpha u
=
F_2'(u)-F_1'(u)+|u|^{p-2}u
\qquad \text{in }\R^3.
\]

We define
\[
H_\varepsilon
=
\left\{
u\in H^\alpha(\R^3):
\int_{\R^3}V(\varepsilon x)u^2\,dx<\infty
\right\},
\qquad
X_\varepsilon=H_\varepsilon\cap L^{F_1}(\R^3).
\]
On \(H_\varepsilon\) we use the norm
\[
\|u\|_{H_\varepsilon}
=
\left(
\int_{\R^3}
\Bigl(
|(-\Delta)^{\frac{\alpha}{2}}u|^2+(V(\varepsilon x)+1)|u|^2
\Bigr)\,dx
\right)^{\frac12},
\]
and on \(X_\varepsilon\) the norm
\[
\|u\|_\varepsilon=\|u\|_{H_\varepsilon}+\|u\|_{F_1},
\]
where
\[
\|u\|_{F_1}
=
\inf\left\{
\lambda>0:
\int_{\R^3}F_1\!\left(\frac{|u|}{\lambda}\right)\,dx\le1
\right\}.
\]
Since \(V(\varepsilon x)+1\ge V_0+1>0\), there exists \(C>0\), independent of \(\varepsilon\), such that
\[
\|u\|_{H^\alpha(\R^3)}\le C\|u\|_{H_\varepsilon},
\qquad \forall\, u\in H_\varepsilon.
\]
Hence
\[
X_\varepsilon\hookrightarrow H^\alpha(\R^3)
\]
continuously. Moreover, \(H_\varepsilon\) is a Hilbert space and, by Proposition~\ref{prop2.1}, \(L^{F_1}(\R^3)\) is reflexive and separable. Therefore \((X_\varepsilon,\|\cdot\|_\varepsilon)\) is a reflexive and separable Banach space, and
\[
X_\varepsilon\hookrightarrow L^{F_1}(\R^3)
\]
continuously.

Since \(\alpha\in\bigl(\frac34,1\bigr)\), we have
\[
2\le \frac{12}{3+2\alpha}\le 2_\alpha^*,
\]
and thus
\[
H^\alpha(\R^3)\hookrightarrow L^{\frac{12}{3+2\alpha}}(\R^3).
\]
Using H\"older's inequality, the embedding \(\D^{\alpha,2}(\R^3)\hookrightarrow L^{2_\alpha^*}(\R^3)\), and the equation satisfied by \(\phi_u^\alpha\), we obtain
\[
\|\phi_u^\alpha\|_{\D^{\alpha,2}}
\le
C\|u\|_{\frac{12}{3+2\alpha}}^2,
\]
and consequently
\[
\int_{\R^3}\phi_u^\alpha u^2\,dx
\le
C\|u\|_{\frac{12}{3+2\alpha}}^4
\le
C\|u\|_{H^\alpha(\R^3)}^4
\le
C\|u\|_{H_\varepsilon}^4
\le
C\|u\|_\varepsilon^4,
\qquad \forall\, u\in X_\varepsilon.
\]

The associated energy functional is given by
\[
\begin{aligned}
J_\varepsilon(u)
&=
\frac12\int_{\R^3}
\Bigl(
|(-\Delta)^{\frac{\alpha}{2}}u|^2+(V(\varepsilon x)+1)|u|^2
\Bigr)\,dx
+\frac14\int_{\R^3}\phi_u^\alpha u^2\,dx \\
&\quad
+\int_{\R^3}F_1(u)\,dx
-\int_{\R^3}F_2(u)\,dx
-\frac{1}{p}\int_{\R^3}|u|^p\,dx,
\qquad u\in X_\varepsilon.
\end{aligned}
\]
Standard arguments, together with Proposition~\ref{prop2.1}, property \((f_2)\), and the estimate above, show that
\[
J_\varepsilon\in C^1(X_\varepsilon,\R),
\]
and, for every \(u,v\in X_\varepsilon\),
\[
\begin{aligned}
\langle J_\varepsilon'(u),v\rangle
&=
\int_{\R^3}
(-\Delta)^{\frac{\alpha}{2}}u\,(-\Delta)^{\frac{\alpha}{2}}v\,dx
+\int_{\R^3}(V(\varepsilon x)+1)uv\,dx \\
&\quad
+\int_{\R^3}\phi_u^\alpha uv\,dx
+\int_{\R^3}F_1'(u)v\,dx
-\int_{\R^3}F_2'(u)v\,dx
-\int_{\R^3}|u|^{p-2}uv\,dx.
\end{aligned}
\]

\begin{lemma}\label{lem2.2}
Under \((V)\), the set
\[
M=\{x\in\R^3:V(x)=V_0\}
\]
is nonempty and compact.
\end{lemma}

\begin{proof}
Choose \(R>0\) so large that
\[
|x|\ge R\quad\Longrightarrow\quad V(x)>\frac{V_0+V_\infty}{2}>V_0.
\]
Hence every minimizing sequence for \(V\) is contained in \(\overline{B_R(0)}\). By continuity of \(V\), the infimum \(V_0\) is attained, so \(M\neq\emptyset\). Since \(M=V^{-1}(\{V_0\})\) is closed and contained in \(\overline{B_R(0)}\), it is compact.
\end{proof}

For later use, define
\[
G_1(s):=F_1'(s)s,
\qquad
G_2(s):=F_2'(s)s,
\qquad s\in\R.
\]

\begin{lemma}\label{lem2.3}
The functions \(G_1\) and \(G_2\) belong to \(C^1(\R,\R)\). Moreover, the following estimates hold.
\begin{itemize}
    \item[(i)] There exists \(C>0\) such that
    \[
    lF_1(s)\le G_1(s)\le 2F_1(s),
    \qquad
    |G_1'(s)|\le C(1+|s|)
    \qquad \forall\, s\in\R.
    \]
    \item[(ii)] For every \(q\in(2,2_\alpha^*)\) there exists \(C_q>0\) such that
    \[
    |F_2(s)|+|G_2(s)|\le C_q|s|^q,
    \qquad
    |F_2'(s)|+|G_2'(s)|\le C_q\bigl(|s|+|s|^{q-1}\bigr)
    \qquad \forall\, s\in\R.
    \]
\end{itemize}
Consequently, for every \(\varepsilon>0\), the maps
\[
u\mapsto \int_{\R^3}G_1(u)\,dx,
\qquad
u\mapsto \int_{\R^3}G_2(u)\,dx,
\qquad
u\mapsto \int_{\R^3}\phi_u^\alpha u^2\,dx
\]
are of class \(C^1\) on \(X_\varepsilon\), and therefore
\[
\mathcal H_\varepsilon(u):=\langle J_\varepsilon'(u),u\rangle
\]
belongs to \(C^1(X_\varepsilon,\R)\). The same conclusion holds for the autonomous map
\[
\mathcal H_\mu(u):=\langle J_\mu'(u),u\rangle,
\qquad \mu>-1,
\]
on \(X\).
\end{lemma}

\begin{proof}
The explicit formulas for \(F_1'\) and \(F_2'\) show that \(G_1\) and \(G_2\) are \(C^1\) on each of the regions \((-\infty,-\delta_0)\), \((-\delta_0,\delta_0)\), and \((\delta_0,+\infty)\), and that the one-sided derivatives match at \(\pm\delta_0\); hence \(G_1,G_2\in C^1(\R,\R)\). The estimate \(lF_1\le G_1\le 2F_1\) follows from \eqref{eq2.1}, while the bound on \(G_1'\) is immediate from the explicit expression of \(G_1\) on the two pieces \(|s|<\delta_0\) and \(|s|\ge \delta_0\).

For \(F_2\) and \(G_2\), the required bounds follow from the fact that both functions vanish on \((-\delta_0,\delta_0)\) and that, for \(|s|\ge \delta_0\), the logarithmic factors can be absorbed by any power \(|s|^{q-2}\) with \(q>2\). The differentiability of the associated integral functionals is then a standard consequence of Nemytskii operator theory, together with the embeddings
\[
X_\varepsilon\hookrightarrow H^\alpha(\R^3)\hookrightarrow L^2(\R^3)\cap L^q(\R^3),
\qquad q\in(2,2_\alpha^*),
\]
and Lemma~\ref{lem2.5}(ii)--(vi) for the Poisson term.
\end{proof}

\begin{lemma}\label{lem2.4}
Let \(Y\) denote either \(X\) or \(X_\varepsilon\). Let \(\{u_n\}\subset Y\) be bounded, assume that
\[
u_n\rightharpoonup u\quad\text{in }H^\alpha(\R^3),
\qquad
u_n(x)\to u(x)\quad\text{a.e. in }\R^3,
\]
and set \(z_n:=u_n-u\). Then the following splittings hold:
\[
\int_{\R^3}F_1(u_n)\,dx
=
\int_{\R^3}F_1(z_n)\,dx
+
\int_{\R^3}F_1(u)\,dx
+o(1),
\]
\[
\int_{\R^3}G_1(u_n)\,dx
=
\int_{\R^3}G_1(z_n)\,dx
+
\int_{\R^3}G_1(u)\,dx
+o(1),
\]
\[
\int_{\R^3}F_2(u_n)\,dx
=
\int_{\R^3}F_2(z_n)\,dx
+
\int_{\R^3}F_2(u)\,dx
+o(1),
\]
\[
\int_{\R^3}G_2(u_n)\,dx
=
\int_{\R^3}G_2(z_n)\,dx
+
\int_{\R^3}G_2(u)\,dx
+o(1),
\]
and, for every \(r\in[2,2_\alpha^*)\),
\[
\int_{\R^3}|u_n|^r\,dx
=
\int_{\R^3}|z_n|^r\,dx
+
\int_{\R^3}|u|^r\,dx
+o(1).
\]
Moreover, by Lemma~\ref{lem2.5}(v),
\[
\int_{\R^3}\phi_{u_n}^\alpha u_n^2\,dx
=
\int_{\R^3}\phi_{z_n}^\alpha z_n^2\,dx
+
\int_{\R^3}\phi_u^\alpha u^2\,dx
+o(1).
\]
The same statements remain true after replacing \(u_n\) by translated sequences \(u_n(\cdot+y_n)\), with \(\{y_n\}\subset\R^3\) arbitrary.
\end{lemma}

\begin{proof}
The \(L^r\)-splitting is exactly the classical Br\'ezis--Lieb lemma. The Poisson splitting is Lemma~\ref{lem2.5}(v). For \(F_2\) and \(G_2\), Lemma~\ref{lem2.3}(ii) reduces the claim to the classical Br\'ezis--Lieb lemma with any exponent \(q\in(2,2_\alpha^*)\).

For \(F_1\) and \(G_1\), we use the Br\'ezis--Lieb theorem in Orlicz spaces. Proposition~\ref{prop2.1} shows that \(F_1\) is an \(N\)-function satisfying the \(\Delta_2\)-condition, and Lemma~\ref{lem2.3}(i) shows that \(G_1\) is equivalent to \(F_1\). Since \(\{u_n\}\) is bounded in \(L^{F_1}(\R^3)\), the Orlicz-space Br\'ezis--Lieb lemma applies to both \(F_1\) and \(G_1\), giving the desired decompositions. The translation statement follows from the translation invariance of the underlying norms and integrals.
\end{proof}

We shall also use the following standard properties of \(\phi_u^\alpha\).

\begin{lemma}\label{lem2.5}
{\rm\cite[Lemma 2.3]{Teng2016}}
Let \(u\in H^\alpha(\R^3)\) with \(\alpha\in\left[\frac12,1\right)\). Then:
\begin{itemize}
    \item[(i)] \(\phi_u^\alpha\ge0\);
    \item[(ii)] the map
    \[
    u\mapsto \phi_u^\alpha:H^\alpha(\R^3)\to \D^{\alpha,2}(\R^3)
    \]
    is continuous and maps bounded sets into bounded sets;
    \item[(iii)]
    \[
    \int_{\R^3}\phi_u^\alpha u^2\,dx
    \le
    C\|u\|_{\frac{12}{3+2\alpha}}^4
    \le
    C\|u\|_{H^\alpha(\R^3)}^4;
    \]
    \item[(iv)] for every \(\tau\in\R\),
    \[
    \phi_{\tau u}^\alpha=\tau^2\phi_u^\alpha;
    \]
    \item[(v)] if \(u_n\rightharpoonup u\) in \(H^\alpha(\R^3)\), then \(\phi_{u_n}^\alpha\rightharpoonup \phi_u^\alpha\) in \(\D^{\alpha,2}(\R^3)\), and
    \[
    \int_{\R^3}\phi_{u_n}^\alpha u_n^2\,dx
    =
    \int_{\R^3}\phi_{u_n-u}^\alpha (u_n-u)^2\,dx
    +
    \int_{\R^3}\phi_u^\alpha u^2\,dx
    +
    o(1);
    \]
    \item[(vi)] if \(u_n\to u\) in \(H^\alpha(\R^3)\), then \(\phi_{u_n}^\alpha\to \phi_u^\alpha\) in \(\D^{\alpha,2}(\R^3)\), and
    \[
    \int_{\R^3}\phi_{u_n}^\alpha u_n^2\,dx
    \to
    \int_{\R^3}\phi_u^\alpha u^2\,dx.
    \]
\end{itemize}
\end{lemma}

\begin{lemma}\label{lem2.6}
{\rm\cite{MR3216834}}
Let \(u\in \D^{\alpha,2}(\R^3)\) and \(\varphi\in C_0^\infty(\R^3)\). For \(r>0\), define
\[
\varphi_r(x)=\varphi\!\left(\frac{x}{r}\right).
\]
Then:
\begin{itemize}
    \item[(i)] \(u\varphi_r\to0\) in \(\D^{\alpha,2}(\R^3)\) as \(r\to0\);
    \item[(ii)] if \(\varphi\equiv1\) in a neighbourhood of the origin, then
    \[
    u\varphi_r\to u
    \qquad \text{in } \D^{\alpha,2}(\R^3)
    \qquad \text{as } r\to+\infty.
    \]
\end{itemize}
\end{lemma}

\section{Nehari manifold and ground states}

\subsection{Geometry of the Nehari manifold}

We define the Nehari manifold associated with \(J_{\varepsilon}\) by
\[
\mathcal{N}_{\varepsilon}
=
\left\{
u\in X_{\varepsilon}\setminus\{0\}:
\langle J_{\varepsilon}'(u),u\rangle=0
\right\}.
\]
Every nontrivial critical point of \(J_{\varepsilon}\) belongs to \(\mathcal{N}_{\varepsilon}\).

\begin{lemma}\label{lem3.1}
There exists \(\beta>0\), independent of \(\varepsilon\), such that
\[
\|u\|_{\varepsilon}\ge \|u\|_{H_{\varepsilon}}\ge \beta,
\qquad \forall\, u\in \mathcal{N}_{\varepsilon}.
\]
\end{lemma}

\begin{proof}
Let \(u\in\mathcal N_\varepsilon\). Then
\[
\|u\|_{H_{\varepsilon}}^{2}
+\int_{\R^{3}}F_{1}'(u)u\,dx
+\int_{\R^{3}}\phi_{u}^{\alpha}u^{2}\,dx
=
\int_{\R^{3}}F_{2}'(u)u\,dx
+
\int_{\R^{3}}|u|^{p}\,dx.
\]
By \((f_1)\) and Lemma~\ref{lem2.5}(i),
\[
\|u\|_{H_{\varepsilon}}^{2}
\le
\int_{\R^{3}}|F_{2}'(u)u|\,dx
+
\int_{\R^{3}}|u|^{p}\,dx.
\]
Using \((f_2)\) with \(r=p\), we obtain
\[
\|u\|_{H_{\varepsilon}}^{2}\le C\int_{\R^{3}}|u|^{p}\,dx.
\]
Since \(V(\varepsilon x)+1\ge V_0+1>0\), the \(H^\alpha\)-norm is uniformly controlled by \(\|\cdot\|_{H_\varepsilon}\). Hence, by the Sobolev embedding,
\[
\|u\|_{p}\le C\|u\|_{H_{\varepsilon}},
\]
and therefore
\[
\|u\|_{H_{\varepsilon}}^{2}\le C\|u\|_{H_{\varepsilon}}^{p}.
\]
Because \(p>2\) and \(u\neq0\), this yields \(\|u\|_{H_\varepsilon}\ge\beta\) for some \(\beta>0\) independent of \(\varepsilon\). The inequality \(\|u\|_\varepsilon\ge \|u\|_{H_\varepsilon}\) is immediate.
\end{proof}

\begin{lemma}\label{lem3.2}
The following statements hold:
\begin{itemize}
	\item[(i)] there exist \(r,\rho>0\), independent of \(\varepsilon\), such that
	\[
	J_{\varepsilon}(u)\ge \rho,
	\qquad
	\forall\, u\in X_{\varepsilon}\ \text{with }\|u\|_{\varepsilon}=r;
	\]
	
	\item[(ii)] for every \(\varepsilon>0\), there exists \(v\in X_{\varepsilon}\) with \(\|v\|_{\varepsilon}>r\) such that
	\[
	J_{\varepsilon}(v)<0=J_{\varepsilon}(0).
	\]
\end{itemize}
\end{lemma}

\begin{proof}
(i) By Proposition~\ref{prop2.1} and \eqref{eq2.2},
\[
\int_{\R^{3}}F_{1}(u)\,dx
\ge
\xi_{0}\bigl(\|u\|_{F_{1}}\bigr),
\qquad
\xi_{0}(t)=\min\{t^{l},t^{2}\},
\]
for some \(l\in(1,2)\). Moreover, by \((f_2)\) and \(p>2\), there exists \(C>0\) such that
\[
|F_{2}(s)|\le C|s|^{p},
\qquad \forall\, s\in\R.
\]
Using the continuous embedding
\[
X_{\varepsilon}\hookrightarrow H^{\alpha}(\R^{3})\hookrightarrow L^{p}(\R^{3}),
\]
with constants independent of \(\varepsilon\), we obtain
\[
\int_{\R^{3}}|F_{2}(u)|\,dx+\int_{\R^{3}}|u|^{p}\,dx
\le
C\|u\|_{\varepsilon}^{p}.
\]
Hence
\[
J_{\varepsilon}(u)
\ge
\frac12\|u\|_{H_{\varepsilon}}^{2}
+\xi_{0}\bigl(\|u\|_{F_{1}}\bigr)
-C\|u\|_{\varepsilon}^{p}.
\]

Now let \(\|u\|_{\varepsilon}=r\), and write
\[
a=\|u\|_{H_{\varepsilon}},
\qquad
b=\|u\|_{F_{1}}.
\]
Since \(a+b=r\), either \(a\ge r/2\) or \(b\ge r/2\). In the first case,
\[
J_{\varepsilon}(u)\ge \frac12\left(\frac r2\right)^{2}-Cr^{p}.
\]
In the second case, if \(r\in(0,1)\), then \(\xi_0(b)\ge (r/2)^2\), and thus
\[
J_{\varepsilon}(u)\ge \left(\frac r2\right)^{2}-Cr^{p}.
\]
Since \(p>2\), choosing \(r>0\) sufficiently small yields \(J_{\varepsilon}(u)\ge \rho>0\).

(ii) Fix \(u\in X_{\varepsilon}\setminus\{0\}\). Since
\[
F_{1}(s)-F_{2}(s)=-\frac12 s^{2}\log s^{2},
\]
we have
\[
\begin{aligned}
J_{\varepsilon}(tu)
&=
\frac{t^{2}}{2}\|u\|_{H_{\varepsilon}}^{2}
+\frac{t^{4}}{4}\int_{\R^{3}}\phi_{u}^{\alpha}u^{2}\,dx
-\frac{t^{2}}{2}\int_{\R^{3}}u^{2}\log u^{2}\,dx \\
&\quad
-\frac{t^{2}}{2}\log t^{2}\int_{\R^{3}}u^{2}\,dx
-\frac{t^{p}}{p}\int_{\R^{3}}|u|^{p}\,dx.
\end{aligned}
\]
Since \(p>4\), the negative \(t^p\)-term dominates as \(t\to+\infty\), and therefore
\[
J_{\varepsilon}(tu)\to-\infty
\qquad \text{as } t\to+\infty.
\]
Taking \(t\) sufficiently large, we obtain \(v=tu\) with \(\|v\|_{\varepsilon}>r\) and \(J_{\varepsilon}(v)<0\).
\end{proof}

Thus \(J_{\varepsilon}\) has the mountain-pass geometry. We set
\[
\Gamma_{\varepsilon}
=
\left\{
\gamma\in C([0,1],X_{\varepsilon}):
\gamma(0)=0,\ J_{\varepsilon}(\gamma(1))<0
\right\},
\]
and
\[
c_{\varepsilon}
=
\inf_{\gamma\in\Gamma_{\varepsilon}}
\max_{t\in[0,1]}J_{\varepsilon}(\gamma(t)).
\]

\begin{lemma}\label{lem3.3}
For every \(u\in X_{\varepsilon}\setminus\{0\}\), there exists a unique \(t_{u}>0\) such that
\[
t_{u}u\in \mathcal{N}_{\varepsilon},
\qquad
J_{\varepsilon}(t_{u}u)=\max_{t\ge0}J_{\varepsilon}(tu).
\]
\end{lemma}

\begin{proof}
Fix \(u\in X_{\varepsilon}\setminus\{0\}\) and set
\[
h(t)=J_{\varepsilon}(tu),
\qquad t\ge0.
\]
By Lemma~\ref{lem3.2},
\[
h(0)=0,
\qquad
h(t)>0 \text{ for } t>0 \text{ small},
\qquad
h(t)\to-\infty \text{ as } t\to+\infty.
\]
Hence \(h\) attains its maximum at some \(t_u>0\), and \(h'(t_u)=0\). Thus
\[
\langle J_{\varepsilon}'(t_uu),u\rangle=0,
\]
so \(t_uu\in\mathcal N_\varepsilon\).

We now prove uniqueness. Suppose that \(0<t_1<t_2\) and
\[
h'(t_1)=h'(t_2)=0.
\]
Then, for \(i=1,2\),
\[
\|u\|_{H_{\varepsilon}}^{2}
+\int_{\R^{3}}\frac{F_{1}'(t_{i}u)}{t_{i}u}u^{2}\,dx
=
\int_{\R^{3}}\frac{F_{2}'(t_{i}u)}{t_{i}u}u^{2}\,dx
+t_{i}^{2}\int_{\R^{3}}\phi_{u}^{\alpha}u^{2}\,dx
+t_{i}^{p-2}\int_{\R^{3}}|u|^{p}\,dx,
\]
where the quotients are understood as zero on \(\{u=0\}\).

Now \(s\mapsto F_1'(s)/s\) is even and nonincreasing in \(|s|\), hence the left-hand side is nonincreasing in \(t\). By \((f_3)\), \(s\mapsto F_2'(s)/s\) is even and nondecreasing in \(|s|\), while the remaining two terms on the right-hand side are strictly increasing in \(t\). Thus the right-hand side is strictly increasing in \(t\), which is impossible. Hence \(t_u\) is unique.
\end{proof}

\begin{lemma}\label{lem3.4}
The set \(\mathcal{N}_{\varepsilon}\) is a \(C^{1}\)-manifold for every \(\varepsilon>0\).
\end{lemma}

\begin{proof}
Define
\[
\mathcal{H}_{\varepsilon}(u)=\langle J_{\varepsilon}'(u),u\rangle,
\qquad u\in X_{\varepsilon}.
\]
Then
\[
\mathcal N_\varepsilon
=
\{u\in X_\varepsilon\setminus\{0\}:\mathcal H_\varepsilon(u)=0\}.
\]
Since
\[
\mathcal H_\varepsilon\in C^1(X_\varepsilon,\R),
\]
it suffices to prove that
\[
\langle \mathcal H_\varepsilon'(u),u\rangle\neq0
\qquad \forall\, u\in\mathcal N_\varepsilon.
\]

Fix \(u\in\mathcal N_\varepsilon\), and set
\[
h(t)=J_\varepsilon(tu),
\qquad
\Psi(t)=\frac{h'(t)}{t},
\qquad t>0.
\]
Then
\[
\Psi(t)
=
\|u\|_{H_{\varepsilon}}^{2}
+\int_{\R^{3}}\frac{F_{1}'(tu)}{tu}u^{2}\,dx
-\int_{\R^{3}}\frac{F_{2}'(tu)}{tu}u^{2}\,dx
-t^{2}\int_{\R^{3}}\phi_{u}^{\alpha}u^{2}\,dx
-t^{p-2}\int_{\R^{3}}|u|^{p}\,dx.
\]
By the proof of Lemma~\ref{lem3.3}, \(\Psi\) is strictly decreasing on \((0,+\infty)\). Since \(u\in\mathcal N_\varepsilon\), we have \(h'(1)=0\), hence \(\Psi(1)=0\), and therefore
\[
\Psi'(1)<0.
\]
Moreover,
\[
\mathcal H_\varepsilon(tu)=\langle J_\varepsilon'(tu),tu\rangle=t\,h'(t)=t^{2}\Psi(t).
\]
Differentiating at \(t=1\), we obtain
\[
\langle \mathcal H_\varepsilon'(u),u\rangle
=
\left.\frac{d}{dt}\right|_{t=1}\mathcal H_\varepsilon(tu)
=
2\Psi(1)+\Psi'(1)
=
\Psi'(1)<0.
\]
Thus \(0\) is a regular value of \(\mathcal H_\varepsilon\) on \(X_\varepsilon\setminus\{0\}\), and the Implicit Function Theorem yields the conclusion.
\end{proof}

\subsection{The mountain-pass level and its characterization}

\begin{lemma}\label{lem3.5}
Let \(\{u_n\}\subset X_{\varepsilon}\setminus\{0\}\) be a bounded sequence such that
\[
\langle J_{\varepsilon}'(u_n),u_n\rangle\to 0
\qquad\text{and}\qquad
u_n\to u\neq0
\quad\text{in }X_{\varepsilon}.
\]
For each \(n\), let \(h_n>0\) be the unique number such that
\[
h_nu_n\in \mathcal{N}_{\varepsilon}.
\]
Then
\[
h_n\to1.
\]
\end{lemma}

\begin{proof}
By Lemma~\ref{lem3.3}, \(h_n\) is well defined. We first show that \(\{h_n\}\) is bounded. If \(h_n\to+\infty\) along a subsequence, then, since \(u_n\to u\neq0\) in \(X_\varepsilon\),
\[
\int_{\R^{3}}|u_n|^{p}\,dx\to \int_{\R^{3}}|u|^{p}\,dx>0.
\]
On the other hand, from
\[
\langle J_{\varepsilon}'(h_nu_n),h_nu_n\rangle=0
\]
we obtain
\[
\begin{aligned}
0
&=
h_n^{2}\|u_n\|_{H_{\varepsilon}}^{2}
+h_n^{4}\int_{\R^{3}}\phi_{u_n}^{\alpha}u_n^{2}\,dx
-\int_{\R^{3}}(h_nu_n)^{2}\log (h_nu_n)^{2}\,dx \\
&\quad
-h_n^{2}\int_{\R^{3}}u_n^{2}\,dx
-h_n^{p}\int_{\R^{3}}|u_n|^{p}\,dx.
\end{aligned}
\]
The logarithmic term grows at most like \(h_n^2\log h_n\), whereas the last term is of order \(h_n^p\) with \(p>4\), which is impossible. Thus \(\{h_n\}\) is bounded above.

If \(h_n\to0\), then \(\|h_nu_n\|_{H_\varepsilon}\to0\), contradicting Lemma~\ref{lem3.1} because \(h_nu_n\in\mathcal N_\varepsilon\). Hence \(\{h_n\}\) is also bounded away from \(0\). Passing to a subsequence, we may assume
\[
h_n\to h>0.
\]
Since \(u_n\to u\) in \(X_\varepsilon\) and \(\langle J_\varepsilon'(u_n),u_n\rangle\to0\),
\[
\langle J_\varepsilon'(u),u\rangle=0,
\]
so \(u\in\mathcal N_\varepsilon\). Likewise,
\[
0
=
\lim_{n\to\infty}\langle J_\varepsilon'(h_nu_n),h_nu_n\rangle
=
\langle J_\varepsilon'(hu),hu\rangle,
\]
hence \(hu\in\mathcal N_\varepsilon\). By the uniqueness statement in Lemma~\ref{lem3.3}, we must have \(h=1\). Therefore \(h_n\to1\).
\end{proof}

We next introduce the autonomous functionals
\[
J_{0}(u)
=
\frac12\int_{\R^{3}}
\Bigl(
|(-\Delta)^{\frac{\alpha}{2}}u|^{2}
+
(V_0+1)|u|^{2}
\Bigr)\,dx
+\frac14\int_{\R^{3}}\phi_{u}^{\alpha}u^{2}\,dx
+\int_{\R^{3}}F_{1}(u)\,dx
-\int_{\R^{3}}F_{2}(u)\,dx
-\frac1p\int_{\R^{3}}|u|^{p}\,dx
\]
and
\[
J_{\infty}(u)
=
\frac12\int_{\R^{3}}
\Bigl(
|(-\Delta)^{\frac{\alpha}{2}}u|^{2}
+
(V_\infty+1)|u|^{2}
\Bigr)\,dx
+\frac14\int_{\R^{3}}\phi_{u}^{\alpha}u^{2}\,dx
+\int_{\R^{3}}F_{1}(u)\,dx
-\int_{\R^{3}}F_{2}(u)\,dx
-\frac1p\int_{\R^{3}}|u|^{p}\,dx.
\]
We set
\[
c_0
=
\inf_{u\in X\setminus\{0\}}\max_{t\ge0}J_0(tu),
\qquad
c_\infty
=
\inf_{u\in X\setminus\{0\}}\max_{t\ge0}J_\infty(tu).
\]

\begin{lemma}\label{lem3.6}
The following statements hold:
\begin{itemize}
	\item[(i)] There exists \(\eta_{0}>0\), independent of \(\varepsilon\), such that
	\[
	c_{\varepsilon}\ge \eta_{0},
	\qquad \forall\, \varepsilon>0.
	\]
	
	\item[(ii)]
	\[
	c_{\varepsilon}
	=
	\inf_{u\in\mathcal{N}_{\varepsilon}}J_{\varepsilon}(u)
	=
	\inf_{u\in X_{\varepsilon}\setminus\{0\}}\max_{t\ge0}J_{\varepsilon}(tu).
	\]
	
	\item[(iii)]
	\[
	\lim_{\varepsilon\to0}c_{\varepsilon}=c_{0}.
	\]
	
	\item[(iv)] If \(c_{0}<c_{\infty}\), then
	\[
	c_{\varepsilon}<c_{\infty}
	\qquad \text{for all sufficiently small } \varepsilon>0.
	\]
\end{itemize}
\end{lemma}

\begin{proof}
(i) This follows directly from Lemma~\ref{lem3.2}(i). We set \(\eta_0=\rho\).

(ii) By Lemma~\ref{lem3.3}, for every \(u\in X_{\varepsilon}\setminus\{0\}\),
\[
t_uu\in\mathcal N_\varepsilon
\qquad\text{and}\qquad
J_\varepsilon(t_uu)=\max_{t\ge0}J_\varepsilon(tu),
\]
hence
\[
\inf_{u\in\mathcal N_\varepsilon}J_\varepsilon(u)
=
\inf_{u\in X_\varepsilon\setminus\{0\}}\max_{t\ge0}J_\varepsilon(tu).
\]

To prove that
\[
c_\varepsilon\le \inf_{u\ne0}\max_{t\ge0}J_\varepsilon(tu),
\]
it is enough to consider the path
\[
\gamma(t)=tTt_u u,
\qquad t\in[0,1],
\]
where \(T>1\) is chosen so that \(J_\varepsilon(Tt_uu)<0\).

Conversely, let \(\gamma\in\Gamma_\varepsilon\). Since \(\gamma(0)=0\), we may choose \(s_0\in(0,1)\) such that \(\gamma(s_0)\neq0\) and
\[
\|\gamma(s_0)\|_{H_\varepsilon}<\beta,
\]
where \(\beta\) is given by Lemma~\ref{lem3.1}. By Lemma~\ref{lem3.3}, \(t_{\gamma(s_0)}\gamma(s_0)\in\mathcal N_\varepsilon\), and Lemma~\ref{lem3.1} implies \(t_{\gamma(s_0)}>1\). Hence
\[
\langle J_\varepsilon'(\gamma(s_0)),\gamma(s_0)\rangle>0.
\]
On the other hand, \(J_\varepsilon(\gamma(1))<0\) implies
\[
\langle J_\varepsilon'(\gamma(1)),\gamma(1)\rangle<0.
\]
By continuity, there exists \(s_*\in(s_0,1)\) such that \(\gamma(s_*)\in\mathcal N_\varepsilon\). Therefore
\[
\max_{t\in[0,1]}J_\varepsilon(\gamma(t))
\ge
\inf_{u\in\mathcal N_\varepsilon}J_\varepsilon(u).
\]
Taking the infimum over \(\gamma\in\Gamma_\varepsilon\), we obtain the reverse inequality.

(iii) We first prove
\[
\limsup_{\varepsilon\to0}c_\varepsilon\le c_0.
\]
Fix \(\eta>0\). Choose \(w\in X\setminus\{0\}\) such that
\[
\max_{t\ge0}J_0(tw)\le c_0+\eta.
\]
Let \(\varphi\in C_0^\infty(\R^3)\) satisfy
\[
0\le\varphi\le1,
\qquad
\varphi=1 \text{ in }B_1(0),
\qquad
\varphi=0 \text{ in }B_2^c(0),
\]
and set
\[
\varphi_R(x)=\varphi\!\left(\frac{x}{R}\right),
\qquad
w_R(x)=\varphi_R(x)w(x).
\]
By Lemma~\ref{lem2.6}, the dominated convergence theorem, and density in \(L^{F_1}(\R^3)\), we have
\[
w_R\to w \qquad \text{in }X
\]
as \(R\to+\infty\). Fix \(R\) sufficiently large. Since \(w_R\) has compact support and \(p>4\), there exists \(T>0\), independent of \(\varepsilon\) small, such that
\[
\max_{t\ge0}J_\varepsilon(tw_R)=\max_{t\in[0,T]}J_\varepsilon(tw_R).
\]
Moreover, \(V(\varepsilon x)\to V_0\) uniformly on \(\operatorname{supp}(w_R)\), so
\[
J_\varepsilon(tw_R)\to J_0(tw_R)
\qquad \text{uniformly for } t\in[0,T].
\]
Hence, for \(\varepsilon\) small,
\[
c_\varepsilon
\le
\max_{t\ge0}J_\varepsilon(tw_R)
\le
\max_{t\ge0}J_0(tw_R)+\eta
\le
\max_{t\ge0}J_0(tw)+2\eta
\le
c_0+3\eta.
\]
Thus \(\limsup_{\varepsilon\to0}c_\varepsilon\le c_0\).

On the other hand, since \(V(\varepsilon x)\ge V_0\),
\[
J_\varepsilon(u)\ge J_0(u),
\qquad \forall\, u\in X_\varepsilon,
\]
and therefore, by part (ii),
\[
c_\varepsilon
=
\inf_{u\in X_\varepsilon\setminus\{0\}}\max_{t\ge0}J_\varepsilon(tu)
\ge
\inf_{u\in X\setminus\{0\}}\max_{t\ge0}J_0(tu)
=
c_0.
\]
Hence \(\liminf_{\varepsilon\to0}c_\varepsilon\ge c_0\), and the conclusion follows.

(iv) This is immediate from part (iii).
\end{proof}

\subsection{Compactness at the mountain-pass level}

Lemma~\ref{lem3.2} yields a \((PS)_{c_{\varepsilon}}\)-sequence for \(J_{\varepsilon}\). We now show that, under
\[
c_{\varepsilon}<c_{\infty},
\]
such a sequence is compact in \(L^r(\R^3)\) for every \(r\in[2,2_\alpha^*)\).

\begin{lemma}\label{lem3.7}
Let \(\varepsilon>0\) be such that
\[
c_{\varepsilon}<c_{\infty}.
\]
Let \(\{u_n\}\subset X_{\varepsilon}\) be a \((PS)_{c_{\varepsilon}}\)-sequence for \(J_{\varepsilon}\). Then, up to a subsequence, there exists \(u_{\varepsilon}\in X_{\varepsilon}\setminus\{0\}\) such that
\[
u_n\to u_{\varepsilon}
\qquad \text{in }L^{r}(\R^{3}),
\qquad \forall\, r\in[2,2_{\alpha}^{*}).
\]
\end{lemma}

\begin{proof}
Since \(\{u_n\}\) is bounded in \(X_\varepsilon\), up to a subsequence,
\[
u_n\rightharpoonup u_\varepsilon \text{ in }X_\varepsilon,
\qquad
u_n\rightharpoonup u_\varepsilon \text{ in }H^\alpha(\R^3),
\]
\[
u_n\to u_\varepsilon \text{ in }L^r_{\mathrm{loc}}(\R^3),\quad \forall\, r\in[1,2_\alpha^*),
\qquad
u_n(x)\to u_\varepsilon(x) \text{ a.e. in }\R^3.
\]
Passing to the limit in \(J_\varepsilon'(u_n)\to0\), we obtain
\[
J_\varepsilon'(u_\varepsilon)=0
\qquad \text{in }X_\varepsilon^*.
\]
Set \(z_n:=u_n-u_\varepsilon\). By Lemma~\ref{lem2.4},
\[
J_\varepsilon(u_n)=J_\varepsilon(u_\varepsilon)+J_\varepsilon(z_n)+o(1)
\]
and
\[
\langle J_\varepsilon'(u_n),u_n\rangle
=
\langle J_\varepsilon'(u_\varepsilon),u_\varepsilon\rangle
+
\langle J_\varepsilon'(z_n),z_n\rangle
+o(1),
\]
whence
\[
\langle J_\varepsilon'(z_n),z_n\rangle=o(1).
\]

We first exclude vanishing. If
\[
\sup_{y\in\R^3}\int_{B_R(y)}|u_n|^2\,dx\to0
\qquad \forall\, R>0,
\]
then Lions' lemma gives
\[
u_n\to0
\qquad \text{in }L^r(\R^3),
\qquad \forall\, r\in(2,2_\alpha^*).
\]
Using \(\langle J_\varepsilon'(u_n),u_n\rangle=o_n(1)\), together with \((f_1)\), Lemma~\ref{lem2.5}(i), and \((f_2)\), we infer that
\[
\|u_n\|_{H_\varepsilon}\to0,
\qquad
\int_{\R^3}G_1(u_n)\,dx\to0.
\]
Since \(lF_1\le G_1\) by Lemma~\ref{lem2.3}(i), it follows that
\[
\int_{\R^3}F_1(u_n)\,dx\to0.
\]
Hence \(u_n\to0\) in \(X_\varepsilon\), and therefore \(J_\varepsilon(u_n)\to0\), contradicting
\[
J_\varepsilon(u_n)\to c_\varepsilon\ge \eta_0>0.
\]
Thus \(\{u_n\}\) does not vanish.

We claim that \(z_n\to0\) in \(L^r(\R^3)\) for every \(r\in[2,2_\alpha^*)\). Assume by contradiction that this is false. Then \(\{z_n\}\) does not vanish. By concentration--compactness, there exist \(R,\delta>0\) and a sequence \(\{y_n^1\}\subset\R^3\) such that
\[
\int_{B_R(y_n^1)}|z_n|^2\,dx\ge\delta
\qquad \forall\, n.
\]
Since \(z_n\to0\) in \(L^2_{\mathrm{loc}}(\R^3)\), necessarily \(|y_n^1|\to+\infty\). Define
\[
w_n^1(x):=z_n(x+y_n^1).
\]
Then \(\{w_n^1\}\) is bounded in \(H^\alpha(\R^3)\). Up to a subsequence,
\[
w_n^1\rightharpoonup w^1
\qquad \text{in }H^\alpha(\R^3),
\]
for some \(w^1\neq0\).

For every \(\varphi\in C_c^\infty(\R^3)\), using \(J_\varepsilon'(u_n)\to0\), \(J_\varepsilon'(u_\varepsilon)=0\), the splittings of Lemma~\ref{lem2.4}, the fact that \(u_\varepsilon(\cdot+y_n^1)\to0\) in \(L^r_{\mathrm{loc}}(\R^3)\) for all \(r<2_\alpha^*\), and the local uniform convergence
\[
V\bigl(\varepsilon(\cdot+y_n^1)\bigr)\to V_\infty
\qquad \text{in }L^\infty_{\mathrm{loc}}(\R^3),
\]
we obtain
\[
\langle J_\infty'(w^1),\varphi\rangle=0.
\]
Hence \(w^1\neq0\) is a critical point of \(J_\infty\), so \(w^1\in\mathcal N_\infty\) and
\[
J_\infty(w^1)\ge c_\infty.
\]
Set
\[
z_n^1(x):=z_n(x)-w^1(x-y_n^1).
\]
Again by Lemma~\ref{lem2.4}, by Lemma~\ref{lem2.5}(v), and by dominated convergence for the potential term,
\[
J_\varepsilon(z_n)=J_\infty(w^1)+J_\varepsilon(z_n^1)+o(1)
\]
and
\[
\langle J_\varepsilon'(z_n),z_n\rangle
=
\langle J_\infty'(w^1),w^1\rangle
+
\langle J_\varepsilon'(z_n^1),z_n^1\rangle
+o(1).
\]
If \(z_n^1\) vanishes, then \(z_n^1\to0\) in \(L^r(\R^3)\) for all \(r\in(2,2_\alpha^*)\), and the last identity yields \(\|z_n^1\|_{H_\varepsilon}\to0\); hence \(J_\varepsilon(z_n^1)\to0\). If \(z_n^1\) does not vanish, we repeat the extraction procedure. Since every nontrivial extracted profile belongs to \(\mathcal N_\infty\), Lemma~\ref{lem3.1} applied to the autonomous functional shows that each profile has \(H_\infty\)-norm bounded away from zero. Because \(\{u_n\}\) is bounded, only finitely many profiles can occur.

Therefore, after finitely many steps, we obtain nontrivial critical points \(w^1,\dots,w^k\in X\), pairwise diverging translation sequences \(\{y_n^j\}\), and a remainder \(r_n^k\) vanishing in \(L^r(\R^3)\) for every \(r\in(2,2_\alpha^*)\), such that
\[
J_\varepsilon(z_n)=\sum_{j=1}^k J_\infty(w^j)+o(1).
\]
Consequently,
\[
c_\varepsilon
=
\lim_{n\to\infty}J_\varepsilon(u_n)
=
J_\varepsilon(u_\varepsilon)+\sum_{j=1}^k J_\infty(w^j)
\ge
J_\varepsilon(u_\varepsilon)+c_\infty.
\]
If \(u_\varepsilon=0\), then \(c_\varepsilon\ge c_\infty\), contradicting the assumption. If \(u_\varepsilon\neq0\), then \(u_\varepsilon\in\mathcal N_\varepsilon\), and Lemma~\ref{lem3.6}(ii) gives \(J_\varepsilon(u_\varepsilon)\ge c_\varepsilon\), again a contradiction.

Hence no profile at infinity can occur, and \(z_n\to0\) in \(L^r(\R^3)\) for every \(r\in[2,2_\alpha^*)\). This proves
\[
u_n\to u_\varepsilon
\qquad \text{in }L^r(\R^3),
\qquad \forall\, r\in[2,2_\alpha^*).
\]
Finally, \(u_\varepsilon\neq0\); otherwise \(u_n\to0\) in \(L^r(\R^3)\) for all \(r\in(2,2_\alpha^*)\), and the vanishing argument above would imply \(u_n\to0\) in \(X_\varepsilon\), a contradiction.
\end{proof}

\begin{lemma}\label{lem3.8}
Assume that
\[
c_{\varepsilon}<c_{\infty}.
\]
Then \(J_{\varepsilon}\) satisfies the \((PS)_{c_{\varepsilon}}\) condition.
\end{lemma}

\begin{proof}
Let \(\{u_n\}\subset X_{\varepsilon}\) be a \((PS)_{c_{\varepsilon}}\)-sequence. By Lemma~\ref{lem3.7}, up to a subsequence,
\[
u_n\rightharpoonup u_{\varepsilon}\quad \text{in }X_{\varepsilon},
\qquad
u_n\to u_{\varepsilon}\quad \text{in }L^{r}(\R^{3}),
\qquad \forall\, r\in[2,2_{\alpha}^{*}),
\]
for some \(u_{\varepsilon}\in X_{\varepsilon}\setminus\{0\}\), and
\[
J_{\varepsilon}'(u_{\varepsilon})=0
\qquad \text{in }X_{\varepsilon}^{*}.
\]

From
\[
\langle J_{\varepsilon}'(u_n)-J_{\varepsilon}'(u_{\varepsilon}),u_n-u_{\varepsilon}\rangle=o_n(1)
\]
we obtain
\[
\begin{aligned}
&\|u_n-u_{\varepsilon}\|_{H_{\varepsilon}}^{2}
+\int_{\R^{3}}
\bigl(F_{1}'(u_n)-F_{1}'(u_{\varepsilon})\bigr)(u_n-u_{\varepsilon})\,dx \\
&=
\int_{\R^{3}}
\bigl(F_{2}'(u_n)-F_{2}'(u_{\varepsilon})\bigr)(u_n-u_{\varepsilon})\,dx \\
&\quad
+\int_{\R^{3}}
\Bigl(|u_n|^{p-2}u_n-|u_{\varepsilon}|^{p-2}u_{\varepsilon}\Bigr)(u_n-u_{\varepsilon})\,dx \\
&\quad
-\int_{\R^{3}}
\bigl(\phi_{u_n}^{\alpha}u_n-\phi_{u_{\varepsilon}}^{\alpha}u_{\varepsilon}\bigr)(u_n-u_{\varepsilon})\,dx
+o_n(1).
\end{aligned}
\]
Since \(F_1\) is convex,
\[
\int_{\R^{3}}
\bigl(F_{1}'(u_n)-F_{1}'(u_{\varepsilon})\bigr)(u_n-u_{\varepsilon})\,dx\ge0.
\]
The \(F_2\)-term and the \(p\)-term tend to \(0\) by the strong convergence in \(L^p(\R^3)\). By Lemma~\ref{lem2.5}(ii) and the strong convergence in \(L^{\frac{12}{3+2\alpha}}(\R^3)\), the nonlocal term also tends to \(0\). Hence
\[
\|u_n-u_{\varepsilon}\|_{H_{\varepsilon}}\to0.
\]

It remains to prove the convergence in \(L^{F_1}(\R^3)\). Since \(\{u_n-u_\varepsilon\}\) is bounded in \(L^{F_1}(\R^3)\), \(F_1\in(\Delta_2)\), and \(u_n-u_\varepsilon\to0\) a.e. in \(\R^3\), Vitali's theorem yields
\[
\int_{\R^3}F_1(u_n-u_\varepsilon)\,dx\to0.
\]
By the modular characterization of the Luxemburg norm,
\[
u_n\to u_\varepsilon
\qquad \text{in }L^{F_1}(\R^3).
\]
Therefore
\[
u_n\to u_\varepsilon
\qquad \text{in }X_\varepsilon,
\]
and the proof is complete.
\end{proof}

\subsection{Existence of a ground state solution}

\begin{theorem}\label{thm3.9}
Assume that
\[
c_{0}<c_{\infty}.
\]
Then there exists \(\varepsilon_{0}>0\) such that, for every \(\varepsilon\in(0,\varepsilon_{0})\), the functional \(J_{\varepsilon}\) possesses a nontrivial critical point \(u_{\varepsilon}\in X_{\varepsilon}\). Moreover,
\[
J_{\varepsilon}(u_{\varepsilon})=c_{\varepsilon}
=
\inf_{u\in \mathcal{N}_{\varepsilon}}J_{\varepsilon}(u),
\]
so that \(u_{\varepsilon}\) is a ground state solution of the rescaled problem.
\end{theorem}

\begin{proof}
By Lemma~\ref{lem3.6}(iv), there exists \(\varepsilon_0>0\) such that
\[
c_\varepsilon<c_\infty,
\qquad \forall\, \varepsilon\in(0,\varepsilon_0).
\]
Fix \(\varepsilon\in(0,\varepsilon_0)\). By Lemma~\ref{lem3.2}, \(J_\varepsilon\) has the mountain-pass geometry. Hence the Mountain Pass Theorem yields a \((PS)_{c_\varepsilon}\)-sequence \(\{u_n\}\subset X_\varepsilon\) such that
\[
J_\varepsilon(u_n)\to c_\varepsilon,
\qquad
J_\varepsilon'(u_n)\to0
\quad \text{in }X_\varepsilon^*.
\]
Since \(c_\varepsilon<c_\infty\), Lemma~\ref{lem3.8} implies, up to a subsequence,
\[
u_n\to u_\varepsilon
\qquad \text{in }X_\varepsilon
\]
for some \(u_\varepsilon\in X_\varepsilon\setminus\{0\}\). Therefore
\[
J_\varepsilon'(u_\varepsilon)=0
\qquad\text{and}\qquad
J_\varepsilon(u_\varepsilon)=c_\varepsilon.
\]
By Lemma~\ref{lem3.6}(ii),
\[
c_\varepsilon
=
\inf_{u\in\mathcal N_\varepsilon}J_\varepsilon(u).
\]
Hence \(u_\varepsilon\) is a ground state solution of the rescaled problem.
\end{proof}

\section{Autonomous problem and concentration tools}

Throughout this section, we work in the autonomous space
\[
X:=H^{\alpha}(\R^{3})\cap L^{F_{1}}(\R^{3}).
\]

\subsection{The autonomous family and ground states}

For every \(\mu>-1\), define
\[
J_{\mu}(u)
=
\frac12\int_{\R^{3}}
\Bigl(
|(-\Delta)^{\frac{\alpha}{2}}u|^{2}
+
(\mu+1)|u|^{2}
\Bigr)\,dx
+
\frac14\int_{\R^{3}}\phi_{u}^{\alpha}u^{2}\,dx
+
\int_{\R^{3}}F_{1}(u)\,dx
-
\int_{\R^{3}}F_{2}(u)\,dx
-
\frac1p\int_{\R^{3}}|u|^{p}\,dx,
\]
and
\[
\mathcal N_{\mu}
=
\left\{
u\in X\setminus\{0\}:
\langle J_{\mu}'(u),u\rangle=0
\right\},
\qquad
c_{\mu}
=
\inf_{u\in\mathcal N_{\mu}}J_{\mu}(u).
\]
In particular,
\[
c_0=c_{V_0},
\qquad
c_\infty=c_{V_\infty}.
\]

Arguing exactly as in Section~3, one sees that for every \(\mu>-1\) and every \(u\in X\setminus\{0\}\), there exists a unique \(t_\mu(u)>0\) such that
\[
t_\mu(u)u\in\mathcal N_\mu,
\qquad
J_\mu(t_\mu(u)u)=\max_{t\ge0}J_\mu(tu).
\]
Moreover, \(\mathcal N_\mu\) is a \(C^1\)-manifold, a natural constraint for \(J_\mu\), and
\[
c_\mu
=
\inf_{u\in X\setminus\{0\}}\max_{t\ge0}J_\mu(tu)
>0.
\]

\begin{lemma}\label{lem4.1}
Let \(\mu>-1\), and let \(\{u_n\}\subset \mathcal N_{\mu}\) satisfy
\[
J_{\mu}(u_n)\to c_{\mu}.
\]
Then, up to a subsequence, one of the following alternatives holds:
\begin{itemize}
\item[(i)] \(u_n\to u\) strongly in \(X\);
\item[(ii)] there exists a sequence \(\{y_n\}\subset\R^{3}\) such that
\[
w_n(x):=u_n(x+y_n)
\]
converges strongly in \(X\).
\end{itemize}
In both cases, the limit is a nontrivial critical point of \(J_{\mu}\) at level \(c_{\mu}\). In particular, \(c_{\mu}\) is attained for every \(\mu>-1\).
\end{lemma}

\begin{proof}
Set
\[
\mathcal H_\mu(u)=\langle J_\mu'(u),u\rangle.
\]
Since \(\{u_n\}\subset\mathcal N_\mu\) and \(J_\mu(u_n)\to c_\mu\), the sequence \(\{u_n\}\) is bounded in \(X\).

Applying Ekeland's variational principle on the complete metric space \(\mathcal N_\mu\), we find \(\{v_n\}\subset\mathcal N_\mu\) such that
\[
J_\mu(v_n)\to c_\mu,
\qquad
\|u_n-v_n\|_X\to0,
\qquad
\|(J_\mu|_{\mathcal N_\mu})'(v_n)\|\to0.
\]
Hence there exists \(\lambda_n\in\R\) such that
\[
J_\mu'(v_n)-\lambda_n\mathcal H_\mu'(v_n)=o_n(1)
\qquad\text{in }X^*.
\]
Exactly as in Lemma~\ref{lem3.4}, the fibering-map argument gives
\[
\langle \mathcal H_\mu'(u),u\rangle<0
\qquad \forall\, u\in\mathcal N_\mu.
\]
More precisely, if
\[
h_u(t):=J_\mu(tu),
\qquad u\in\mathcal N_\mu,
\]
then
\[
\langle \mathcal H_\mu'(u),u\rangle=h_u''(1)
\le -2\int_{\R^3}\phi_u^\alpha u^2\,dx-(p-2)\int_{\R^3}|u|^p\,dx.
\]
By the autonomous analogue of Lemma~\ref{lem3.1}, there exist constants \(\beta_\mu,C_\mu>0\) such that
\[
\|u\|_{H_\mu}\ge\beta_\mu,
\qquad
\|u\|_{H_\mu}^2\le C_\mu\int_{\R^3}|u|^p\,dx
\qquad \forall\, u\in\mathcal N_\mu.
\]
Consequently,
\[
\langle \mathcal H_\mu'(u),u\rangle
\le -\frac{(p-2)\beta_\mu^2}{C_\mu}
=:-\sigma_\mu
\qquad \forall\, u\in\mathcal N_\mu.
\]
Testing the previous relation against \(v_n\), and using \(\mathcal H_\mu(v_n)=0\), we obtain \(\lambda_n\to0\). Therefore
\[
J_\mu'(v_n)\to0
\qquad \text{in }X^*.
\]
Thus \(\{v_n\}\) is a free \((PS)_{c_\mu}\)-sequence for \(J_\mu\).

Vanishing is impossible. Indeed, if
\[
\sup_{y\in\R^{3}}\int_{B_R(y)}|v_n|^{2}\,dx\to0
\qquad \forall\, R>0,
\]
then Lions' lemma yields
\[
v_n\to0
\qquad \text{in }L^r(\R^3),
\qquad \forall\, r\in(2,2_\alpha^*).
\]
Using \(J_\mu'(v_n)\to0\), \((f_2)\), and the positivity of the \(F_1\)- and Poisson terms, we infer that
\[
\|v_n\|_{H_\mu}\to0,
\qquad
\int_{\R^3}G_1(v_n)\,dx\to0.
\]
Since \(lF_1\le G_1\), it follows that \(\int_{\R^3}F_1(v_n)\,dx\to0\), hence \(v_n\to0\) in \(X\). This contradicts \(J_\mu(v_n)\to c_\mu>0\).

We now apply the standard translation-invariant profile decomposition based on Lemma~\ref{lem2.4}. There exist an integer \(k\ge1\), nontrivial critical points \(w^1,\dots,w^k\in X\) of \(J_\mu\), and translation sequences \(\{y_n^j\}\subset\R^3\) satisfying \(|y_n^i-y_n^j|\to\infty\) for \(i\neq j\), such that
\[
v_n-\sum_{j=1}^k w^j(\cdot-y_n^j)\to0
\qquad \text{in }L^r(\R^3),
\qquad \forall\, r\in[2,2_\alpha^*),
\]
and
\[
c_\mu
=
\lim_{n\to\infty}J_\mu(v_n)
=
\sum_{j=1}^k J_\mu(w^j).
\]
Since every nontrivial critical point of \(J_\mu\) belongs to \(\mathcal N_\mu\), each \(w^j\) satisfies
\[
J_\mu(w^j)\ge c_\mu.
\]
Hence necessarily \(k=1\). Denote the unique profile by \(w\).

If \(\{y_n^1\}\) is bounded, then, after passing to a subsequence, \(y_n^1\to y_0\). Replacing \(w(\cdot-y_n^1)\) by the translated profile \(w(\cdot-y_0)\), we obtain
\[
v_n\to w(\cdot-y_0)
\qquad \text{in }L^r(\R^3),
\qquad \forall\, r\in[2,2_\alpha^*),
\]
which yields alternative \((i)\).

If \(|y_n^1|\to\infty\), define
\[
w_n(x):=v_n(x+y_n^1).
\]
Then
\[
w_n\to w
\qquad \text{in }L^r(\R^3),
\qquad \forall\, r\in[2,2_\alpha^*),
\]
which yields alternative \((ii)\).

In either case, the strong convergence in \(X\) follows exactly as in the proof of Lemma~\ref{lem3.8}: from
\[
\langle J_\mu'(\xi_n)-J_\mu'(\xi),\xi_n-\xi\rangle=o_n(1),
\]
where \(\xi_n\) denotes either \(v_n\) or \(w_n\), and \(\xi\) denotes the corresponding limit, the convexity of \(F_1\), the strong \(L^p\)-convergence, and Lemma~\ref{lem2.5}(ii),(vi) imply
\[
\xi_n\to\xi
\qquad \text{in }H_\mu.
\]
Since \(\{\xi_n-\xi\}\) is bounded in \(L^{F_1}(\R^3)\), \(F_1\in(\Delta_2)\), and \(\xi_n-\xi\to0\) a.e. in \(\R^3\), the modular characterization of the Luxemburg norm gives
\[
\xi_n\to\xi
\qquad \text{in }L^{F_1}(\R^3).
\]
Therefore \(\xi_n\to\xi\) strongly in \(X\). Finally,
\[
J_\mu(\xi)=\lim_{n\to\infty}J_\mu(v_n)=c_\mu,
\qquad
J_\mu'(\xi)=0.
\]
Since \(\|u_n-v_n\|_X\to0\), the same conclusion holds for the original sequence \(\{u_n\}\).
\end{proof}

\begin{lemma}\label{lem4.2}
For every \(\mu>-1\), the level \(c_{\mu}\) is attained. Moreover, if
\[
-1<\mu_1<\mu_2,
\]
then
\[
c_{\mu_1}<c_{\mu_2}.
\]
In particular,
\[
c_0<c_{\infty}.
\]
\end{lemma}

\begin{proof}
The attainment statement is exactly Lemma~\ref{lem4.1}.

Let \(-1<\mu_1<\mu_2\), and let \(u_{\mu_2}\in\mathcal N_{\mu_2}\) satisfy
\[
J_{\mu_2}(u_{\mu_2})=c_{\mu_2}.
\]
By the mountain-pass characterization,
\[
c_{\mu_1}
=
\inf_{u\in X\setminus\{0\}}\max_{t\ge0}J_{\mu_1}(tu)
\le
\max_{t\ge0}J_{\mu_1}(tu_{\mu_2}).
\]
For every \(t>0\),
\[
J_{\mu_1}(tu_{\mu_2})
=
J_{\mu_2}(tu_{\mu_2})
-
\frac{\mu_2-\mu_1}{2}t^2\int_{\R^{3}}|u_{\mu_2}|^2\,dx
<
J_{\mu_2}(tu_{\mu_2}).
\]
Therefore
\[
\max_{t\ge0}J_{\mu_1}(tu_{\mu_2})
<
\max_{t\ge0}J_{\mu_2}(tu_{\mu_2})
=
J_{\mu_2}(u_{\mu_2})
=
c_{\mu_2},
\]
and hence \(c_{\mu_1}<c_{\mu_2}\).

Since \(V_\infty>V_0>-1\) by \((V)\), we conclude that
\[
c_0=c_{V_0}<c_{V_\infty}=c_\infty.
\]
\end{proof}

\begin{lemma}\label{lem4.3}
Let \(u_0\in X\) be a ground state solution of the limit problem, namely
\[
u_0\in\mathcal N_0,
\qquad
J_0(u_0)=c_0,
\qquad
J_0'(u_0)=0.
\]
Then \(u_0\) can be chosen nonnegative. Moreover, every nonnegative ground state solution is strictly positive in \(\R^3\).
\end{lemma}

\begin{proof}
Let \(u_0\) be any ground state solution. Since
\[
[\,|u_0|\,]_{H^{\alpha}(\R^3)}\le [u_0]_{H^{\alpha}(\R^3)},
\]
and all the other terms in \(J_0\) depend only on \(|u_0|\), we have
\[
J_0(t|u_0|)\le J_0(tu_0),
\qquad \forall\, t\ge0.
\]
Let \(t_0>0\) be such that
\[
t_0|u_0|\in\mathcal N_0
\qquad\text{and}\qquad
J_0(t_0|u_0|)=\max_{t\ge0}J_0(t|u_0|).
\]
Then
\[
c_0
\le
J_0(t_0|u_0|)
\le
\max_{t\ge0}J_0(tu_0)
=
J_0(u_0)
=
c_0.
\]
Hence \(t_0|u_0|\) is also a ground state, and we may assume \(u_0\ge0\).

Since \(u_0\) solves
\[
(-\Delta)^\alpha u_0+V_0u_0+\phi_{u_0}^\alpha u_0
=
u_0\log u_0^2+u_0^{p-1}
\qquad \text{in }\R^3,
\]
standard regularity for fractional equations implies that \(u_0\) is continuous. If \(u_0(x_0)=0\) for some \(x_0\in\R^3\), then \(x_0\) is a global minimum point of \(u_0\). Since \(u_0\ge0\) and \(u_0\not\equiv0\), the pointwise formula for \((-\Delta)^\alpha\) yields
\[
(-\Delta)^\alpha u_0(x_0)<0.
\]
On the other hand, since \(t\log t^2\to0\) as \(t\to0^+\), evaluating the equation at \(x_0\) gives
\[
(-\Delta)^\alpha u_0(x_0)=0,
\]
a contradiction. Therefore
\[
u_0>0
\qquad \text{in }\R^3.
\]
\end{proof}

\subsection{Translation compactness of low-energy Nehari sequences}

\begin{lemma}\label{lem4.4}
Let \(\varepsilon_n\to0^{+}\), and let \(u_n\in\mathcal N_{\varepsilon_n}\) satisfy
\[
J_{\varepsilon_n}(u_n)\to c_0.
\]
Then there exists a sequence \(\{\widetilde y_n\}\subset\R^{3}\) such that, setting
\[
v_n(x)=u_n(x+\widetilde y_n),
\qquad
y_n=\varepsilon_n\widetilde y_n,
\]
one has, up to a subsequence,
\[
v_n\to v
\qquad \text{strongly in }X,
\]
where \(v\neq0\) is a ground state solution of the limit problem, and
\[
y_n\to y_0\in M.
\]
\end{lemma}

\begin{proof}
Since \(u_n\in\mathcal N_{\varepsilon_n}\) and \(J_{\varepsilon_n}(u_n)\to c_0\), the sequence \(\{u_n\}\) is bounded in \(X_{\varepsilon_n}\), hence in \(H^\alpha(\R^3)\).

We first exclude vanishing. If
\[
\sup_{\xi\in\R^3}\int_{B_R(\xi)}|u_n|^2\,dx\to0
\qquad \forall\, R>0,
\]
then, by Lions' lemma,
\[
u_n\to0
\qquad \text{in }L^r(\R^3),
\qquad \forall\, r\in(2,2_\alpha^*).
\]
Using
\[
\langle J_{\varepsilon_n}'(u_n),u_n\rangle=0,
\]
together with \((f_1)\), Lemma~\ref{lem2.5}(i), and \((f_2)\), exactly as in the proof of Lemma~\ref{lem3.7}, we deduce that
\[
\|u_n\|_{H_{\varepsilon_n}}\to0,
\qquad
\int_{\R^3}G_1(u_n)\,dx\to0.
\]
Since \(lF_1\le G_1\), it follows that
\[
\int_{\R^3}F_1(u_n)\,dx\to0,
\]
hence \(u_n\to0\) in \(X_{\varepsilon_n}\). Consequently \(J_{\varepsilon_n}(u_n)\to0\), contradicting \(J_{\varepsilon_n}(u_n)\to c_0>0\).

Therefore vanishing does not occur. Hence there exist \(R>0\), \(\beta>0\), and a sequence \(\{\widetilde y_n\}\subset\R^3\) such that
\[
\int_{B_R(\widetilde y_n)}|u_n|^2\,dx\ge\beta
\qquad \forall\, n.
\]
Define
\[
v_n(x)=u_n(x+\widetilde y_n),
\qquad
y_n=\varepsilon_n\widetilde y_n.
\]
Then
\[
\int_{B_R(0)}|v_n|^2\,dx\ge\beta
\qquad \forall\, n.
\]
Moreover, since \(\{u_n\}\) is bounded in \(X_{\varepsilon_n}\), the translation invariance of the \(H^\alpha\)-norm and of the modular
\[
u\mapsto \int_{\R^3}F_1(u)\,dx
\]
implies that \(\{v_n\}\) is bounded in \(H^\alpha(\R^3)\cap L^{F_1}(\R^3)=X\).

For each \(n\), define
\[
\widehat J_n(w)
=
\frac12\int_{\R^3}
\Bigl(
|(-\Delta)^{\frac{\alpha}{2}}w|^2
+
(V(\varepsilon_n x+y_n)+1)|w|^2
\Bigr)\,dx
+
\frac14\int_{\R^3}\phi_w^\alpha w^2\,dx
+
\int_{\R^3}F_1(w)\,dx
-
\int_{\R^3}F_2(w)\,dx
-
\frac1p\int_{\R^3}|w|^p\,dx.
\]
Then
\[
\widehat J_n(v_n)=J_{\varepsilon_n}(u_n)\to c_0
\]
and
\[
\langle \widehat J_n'(v_n),v_n\rangle
=
\langle J_{\varepsilon_n}'(u_n),u_n\rangle
=0.
\]

Let \(t_n>0\) be the unique number such that
\[
w_n=t_n v_n\in\mathcal N_0.
\]
Since \(V(\varepsilon_n x+y_n)\ge V_0\), we have
\[
J_0(\xi)\le \widehat J_n(\xi)
\qquad \forall\, \xi\in X.
\]
Moreover, \(v_n\in \widehat{\mathcal N}_n\), where
\[
\widehat{\mathcal N}_n
=
\left\{
w\in X\setminus\{0\}:\ \langle \widehat J_n'(w),w\rangle=0
\right\}.
\]
By the uniqueness of the fibering maximum for \(\widehat J_n\),
\[
\widehat J_n(w_n)=\widehat J_n(t_n v_n)\le \widehat J_n(v_n).
\]
Therefore
\[
c_0
\le
J_0(w_n)
\le
\widehat J_n(w_n)
\le
\widehat J_n(v_n)
=
c_0+o_n(1),
\]
and hence
\[
J_0(w_n)\to c_0.
\]

We now show that \(\{t_n\}\) is bounded away from \(0\) and \(+\infty\). Since \(w_n\in\mathcal N_0\) and \(J_0(w_n)\to c_0\), the autonomous analogue of Lemma~\ref{lem3.1} gives
\[
\|w_n\|_{H_0}\ge \beta_0>0,
\]
where
\[
\|u\|_{H_0}^2
=
\int_{\R^3}
\Bigl(
|(-\Delta)^{\frac{\alpha}{2}}u|^2+(V_0+1)|u|^2
\Bigr)\,dx.
\]
Since \(\{v_n\}\) is bounded in \(H^\alpha(\R^3)\), it is bounded in \(H_0\), and therefore
\[
t_n=\frac{\|w_n\|_{H_0}}{\|v_n\|_{H_0}}\ge \underline t>0.
\]

On the other hand, from
\[
\int_{B_R(0)}|v_n|^2\,dx\ge\beta
\]
and the continuous embedding \(H_0\hookrightarrow L^2(B_R(0))\), we obtain
\[
\|v_n\|_{H_0}\ge c_R>0
\qquad \forall\, n.
\]
Since \(\{w_n\}\subset\mathcal N_0\) and \(J_0(w_n)\to c_0\), the same boundedness argument as in Lemma~\ref{lem4.1} shows that \(\{w_n\}\) is bounded in \(X\), hence in \(H_0\). Consequently,
\[
t_n=\frac{\|w_n\|_{H_0}}{\|v_n\|_{H_0}}\le \overline t<+\infty.
\]
Thus, up to a subsequence,
\[
t_n\to t_0>0.
\]

Moreover,
\[
\int_{B_R(0)}|w_n|^2\,dx
=
t_n^2\int_{B_R(0)}|v_n|^2\,dx
\ge
\underline t^{\,2}\beta
\qquad \forall\, n.
\]
Since \(w_n\in\mathcal N_0\) and \(J_0(w_n)\to c_0\), Lemma~\ref{lem4.1} applies. We claim that
\[
w_n\to w
\qquad \text{strongly in }X
\]
for some ground state \(w\) of the limit problem.

Indeed, if alternative \((i)\) of Lemma~\ref{lem4.1} holds, there is nothing to prove. Suppose that alternative \((ii)\) holds. Then there exists \(\{z_n\}\subset\R^3\) such that
\[
\xi_n(x):=w_n(x+z_n)\to w
\qquad \text{strongly in }X.
\]
If \(|z_n|\to+\infty\), then
\[
\int_{B_R(0)}|w_n(x)|^2\,dx
=
\int_{B_R(z_n)}|\xi_n(x)|^2\,dx.
\]
Since \(\xi_n\to w\) strongly in \(L^2(\R^3)\) and \(w\in L^2(\R^3)\), the right-hand side tends to \(0\), which contradicts
\[
\int_{B_R(0)}|w_n|^2\,dx\ge \underline t^{\,2}\beta.
\]
Hence \(\{z_n\}\) is bounded. After passing to a subsequence, \(z_n\to z_0\), and therefore
\[
w_n(x)=\xi_n(x-z_n)\to w(x-z_0)
\qquad \text{strongly in }X.
\]
Renaming the limit, we obtain
\[
w_n\to w
\qquad \text{strongly in }X,
\]
where
\[
w\in\mathcal N_0,
\qquad
J_0(w)=c_0,
\qquad
J_0'(w)=0.
\]

Since \(t_n\to t_0>0\), it follows that
\[
v_n=t_n^{-1}w_n\to \bar v:=t_0^{-1}w
\qquad \text{strongly in }X.
\]
In particular,
\[
\bar v\neq0.
\]

Next we prove that
\[
\int_{\R^3}\bigl(V(\varepsilon_n x+y_n)-V_0\bigr)v_n^2\,dx\to0.
\]
Indeed,
\[
0\le \widehat J_n(w_n)-J_0(w_n)
=
\frac{t_n^2}{2}\int_{\R^3}\bigl(V(\varepsilon_n x+y_n)-V_0\bigr)v_n^2\,dx,
\]
while
\[
\widehat J_n(w_n)\le \widehat J_n(v_n)=c_0+o_n(1)
\qquad\text{and}\qquad
J_0(w_n)\to c_0.
\]
Hence the claim follows.

We now show that \(\{y_n\}\) is bounded. Suppose by contradiction that \(|y_n|\to+\infty\). Since
\[
V(x)\to V_\infty>V_0
\qquad \text{as } |x|\to+\infty,
\]
there exist \(\delta>0\) and \(R_\delta>0\) such that
\[
V(x)\ge V_0+\delta
\qquad \forall\, |x|\ge R_\delta.
\]
For \(n\) large, because \(|y_n|\to+\infty\) and \(\varepsilon_n\to0\), one has
\[
|\varepsilon_n x+y_n|\ge R_\delta
\qquad \forall\, x\in B_R(0),
\]
and therefore
\[
V(\varepsilon_n x+y_n)-V_0\ge\delta
\qquad \forall\, x\in B_R(0).
\]
It follows that
\[
\int_{\R^3}\bigl(V(\varepsilon_n x+y_n)-V_0\bigr)v_n^2\,dx
\ge
\delta\int_{B_R(0)}|v_n|^2\,dx
\ge
\delta\beta,
\]
contradicting the previous limit. Thus \(\{y_n\}\) is bounded. Up to a subsequence,
\[
y_n\to y_0\in\R^3.
\]

We claim that \(y_0\in M\). If \(V(y_0)>V_0\), choose \(\delta>0\) such that
\[
V(y_0)\ge V_0+2\delta.
\]
By continuity of \(V\), there exists \(\rho>0\) such that
\[
V(x)\ge V_0+\delta
\qquad \forall\, x\in B_\rho(y_0).
\]
Since \(y_n\to y_0\) and \(\varepsilon_n\to0\), for \(n\) large we have
\[
\varepsilon_n x+y_n\in B_\rho(y_0)
\qquad \forall\, x\in B_R(0),
\]
and thus
\[
V(\varepsilon_n x+y_n)-V_0\ge\delta
\qquad \forall\, x\in B_R(0).
\]
As above,
\[
\int_{\R^3}\bigl(V(\varepsilon_n x+y_n)-V_0\bigr)v_n^2\,dx
\ge
\delta\beta,
\]
again a contradiction. Therefore
\[
V(y_0)=V_0,
\qquad\text{that is,}\qquad
y_0\in M.
\]

Finally, since
\[
\langle \widehat J_n'(v_n),v_n\rangle=0,
\]
we have
\[
0
=
\langle J_0'(v_n),v_n\rangle
+
\int_{\R^3}\bigl(V(\varepsilon_n x+y_n)-V_0\bigr)v_n^2\,dx.
\]
Passing to the limit, using the strong convergence \(v_n\to\bar v\) in \(X\), we obtain
\[
\langle J_0'(\bar v),\bar v\rangle=0.
\]
Hence \(\bar v\in\mathcal N_0\). Moreover,
\[
J_0(\bar v)
=
\lim_{n\to\infty}
\left[
\widehat J_n(v_n)
-
\frac12\int_{\R^3}\bigl(V(\varepsilon_n x+y_n)-V_0\bigr)v_n^2\,dx
\right]
=
c_0.
\]
Thus \(\bar v\) is a ground state solution of the limit problem.

Since also \(w=t_0\bar v\in\mathcal N_0\), the uniqueness of the fibering maximum on the ray \(\{t\bar v:t>0\}\) implies
\[
t_0=1.
\]
Therefore
\[
v_n\to \bar v
\qquad \text{strongly in }X.
\]
Setting \(v=\bar v\), the proof is complete.
\end{proof}

\subsection{Uniform boundedness and decay estimates}

\begin{lemma}\label{lem4.5}
Let \(\varepsilon_n\to0^{+}\), and let \(u_n\in X_{\varepsilon_n}\) be positive critical points of \(J_{\varepsilon_n}\) such that
\[
J_{\varepsilon_n}(u_n)\to c_0.
\]
Let \(\{\widetilde y_n\}\subset\R^{3}\) be the sequence given by Lemma~\ref{lem4.4}, and set
\[
y_n:=\varepsilon_n\widetilde y_n,
\qquad
v_n(x):=u_n(x+\widetilde y_n).
\]
Then there exists \(C>0\), independent of \(n\), such that
\[
\|v_n\|_{L^\infty(\R^3)}\le C
\qquad \forall\, n.
\]
Moreover,
\[
v_n(x)\to0
\qquad \text{as }|x|\to\infty
\]
uniformly in \(n\).
\end{lemma}

\begin{proof}
By Lemma~\ref{lem4.4},
\[
v_n\to v
\qquad \text{strongly in }X,
\]
where \(v\) is a ground state of the limit problem. In particular, \(\{v_n\}\) is bounded in \(H^\alpha(\R^3)\), and hence
\[
\|v_n\|_{L^r(\R^3)}\le C
\qquad \forall\, r\in[2,2_\alpha^*].
\]

Since each \(u_n\) is a critical point of \(J_{\varepsilon_n}\), the translated function \(v_n\) solves
\[
(-\Delta)^{\alpha}v_n+\bigl(V(\varepsilon_n x+y_n)+1\bigr)v_n+\phi_{v_n}^{\alpha}v_n
=
F_2'(v_n)-F_1'(v_n)+v_n^{p-1}
\qquad \text{in }\R^3.
\]
Because \(v_n\ge0\), \(F_1'(v_n)\ge0\), \(V(\varepsilon_n x+y_n)+1\ge V_0+1>0\), and \(\phi_{v_n}^\alpha\ge0\), we have
\[
(-\Delta)^\alpha v_n
\le
F_2'(v_n)+v_n^{p-1}
\qquad \text{in }\R^3.
\]
By \((f_2)\), there exists \(C>0\) such that
\[
F_2'(t)\le C\bigl(t+t^{p-1}\bigr),
\qquad \forall\, t\ge0,
\]
and hence
\[
(-\Delta)^\alpha v_n
\le
C\bigl(v_n+v_n^{p-1}\bigr)
\qquad \text{in }\R^3.
\]
Since \(p<2_\alpha^*\), the standard fractional Moser iteration gives
\[
\|v_n\|_{L^\infty(\R^3)}\le C
\qquad \forall\, n.
\]

Set
\[
\mu:=\frac{V_0+1}{2}>0,
\qquad
G(t):=\bigl(F_2'(t)-F_1'(t)+t^{p-1}\bigr)^+,
\qquad t\ge0.
\]
Since \(\{v_n\}\) is uniformly bounded in \(L^\infty(\R^3)\), there exists \(M>0\) such that
\[
0\le v_n(x)\le M
\qquad \forall\, x\in\R^3,\ \forall\, n.
\]
Moreover, by the explicit formulas for \(F_1'\) and \(F_2'\), there exists \(\tau\in(0,\delta_0)\) such that
\[
F_2'(t)-F_1'(t)+t^{p-1}<0
\qquad \forall\, t\in(0,\tau].
\]
Thus \(G(t)=0\) on \([0,\tau]\). Since \(G\in C^1([\tau,M])\), it follows that \(G\) is Lipschitz on \([0,M]\). In particular, because \(G(0)=0\), there exists \(L>0\) such that
\[
|G(s)-G(t)|\le L|s-t|,
\qquad
|G(t)|\le Lt,
\qquad \forall\, s,t\in[0,M].
\]

Define
\[
h_n(x):=G(v_n(x)),
\qquad
h(x):=G(v(x)).
\]
Since \(v_n\to v\) in \(L^2(\R^3)\) and \(G\) is Lipschitz on \([0,M]\), we have
\[
h_n\to h
\qquad \text{in }L^2(\R^3).
\]

From the equation for \(v_n\) and the choice of \(\mu\), we obtain
\[
(-\Delta)^\alpha v_n+\mu v_n
=
F_2'(v_n)-F_1'(v_n)+v_n^{p-1}
-\bigl(V(\varepsilon_n x+y_n)+1-\mu\bigr)v_n
-\phi_{v_n}^\alpha v_n
\le h_n
\]
in \(\R^3\), because \(V(\varepsilon_n x+y_n)+1-\mu\ge\mu>0\) and \(\phi_{v_n}^\alpha\ge0\).

Let \(\mathcal K_\mu\) be the Bessel kernel associated with \((-\Delta)^\alpha+\mu\), and define
\[
z_n:=\mathcal K_\mu*h_n,
\qquad
z:=\mathcal K_\mu*h.
\]
Then
\[
(-\Delta)^\alpha z_n+\mu z_n=h_n
\qquad \text{in }\R^3,
\]
and therefore
\[
(-\Delta)^\alpha(z_n-v_n)+\mu(z_n-v_n)\ge0
\qquad \text{in }\R^3.
\]
Testing with \((z_n-v_n)^-\), we obtain
\[
0\le v_n\le z_n
\qquad \text{in }\R^3.
\]

Since \(\alpha>\frac34\), we have \(\mathcal K_\mu\in L^2(\R^3)\). Hence Young's inequality gives
\[
\|z_n-z\|_{L^\infty(\R^3)}
\le
\|\mathcal K_\mu\|_{L^2(\R^3)}\|h_n-h\|_{L^2(\R^3)}
\to0.
\]
Moreover, \(\mathcal K_\mu\in L^2(\R^3)\) and \(h\in L^2(\R^3)\), so \(z=\mathcal K_\mu*h\in C_0(\R^3)\). Therefore
\[
z(x)\to0
\qquad \text{as }|x|\to\infty.
\]

Fix \(\eta>0\). Choose \(R>0\) such that
\[
|z(x)|<\frac{\eta}{2}
\qquad \forall\, |x|\ge R.
\]
Since \(z_n\to z\) in \(L^\infty(\R^3)\), we have
\[
z_n(x)<\eta
\qquad \forall\, |x|\ge R,\ \forall\, n \text{ large}.
\]
Enlarging \(R\) if necessary, the same estimate holds for the finitely many remaining \(n\). Hence
\[
\sup_n\sup_{|x|\ge R}z_n(x)\le\eta.
\]
Since \(0\le v_n\le z_n\), it follows that
\[
\sup_n\sup_{|x|\ge R}v_n(x)\le\eta.
\]
Therefore
\[
v_n(x)\to0
\qquad \text{as }|x|\to\infty
\]
uniformly in \(n\).
\end{proof}

\section{Multiplicity and proof of the main theorem}

\subsection{Low-energy compactness on the Nehari manifold}

In this subsection we study the low-energy compactness of the constrained functional
\[
\left.J_{\varepsilon}\right|_{\mathcal N_{\varepsilon}}.
\]

Recall that
\[
\mathcal H_{\varepsilon}(u):=\langle J_{\varepsilon}'(u),u\rangle,
\qquad
\mathcal N_{\varepsilon}
=
\left\{
u\in X_{\varepsilon}\setminus\{0\}:\mathcal H_{\varepsilon}(u)=0
\right\}.
\]
By Lemma~\ref{lem3.4}, a point \(u\in\mathcal N_\varepsilon\) is a constrained critical point of \(\left.J_\varepsilon\right|_{\mathcal N_\varepsilon}\) if
\[
\|J_{\varepsilon}'(u)\|_{*}
:=
\min_{\lambda\in\R}
\bigl\|J_{\varepsilon}'(u)-\lambda \mathcal H_{\varepsilon}'(u)\bigr\|_{X_{\varepsilon}^{*}}
=0.
\]
Accordingly, a \((PS)_c\)-sequence for \(\left.J_\varepsilon\right|_{\mathcal N_\varepsilon}\) is a sequence \(\{u_n\}\subset\mathcal N_\varepsilon\) such that
\[
J_\varepsilon(u_n)\to c
\qquad\text{and}\qquad
\|J_{\varepsilon}'(u_n)\|_{*}\to0.
\]

\begin{lemma}\label{lem5.1}
Let \(u\in\mathcal N_{\varepsilon}\) be a constrained critical point of \(\left.J_{\varepsilon}\right|_{\mathcal N_{\varepsilon}}\). Then \(u\) is a critical point of \(J_{\varepsilon}\) in \(X_{\varepsilon}\).
\end{lemma}

\begin{proof}
There exists \(\lambda\in\R\) such that
\[
J_{\varepsilon}'(u)=\lambda \mathcal H_{\varepsilon}'(u).
\]
Testing against \(u\), we obtain
\[
0=\langle J_{\varepsilon}'(u),u\rangle
=\lambda \langle \mathcal H_{\varepsilon}'(u),u\rangle.
\]
By Lemma~\ref{lem3.4},
\[
\langle \mathcal H_{\varepsilon}'(u),u\rangle<0.
\]
Hence \(\lambda=0\), and therefore \(J_{\varepsilon}'(u)=0\).
\end{proof}

\begin{lemma}\label{lem5.2}
There exist \(\nu>0\) and \(\varepsilon_{*}>0\) such that, for every
\[
\varepsilon\in(0,\varepsilon_{*}),
\]
the functional \(J_{\varepsilon}\) satisfies the \((PS)_{c}\) condition for every
\[
c\in[c_{\varepsilon},\,c_{0}+\nu].
\]
\end{lemma}

\begin{proof}
By Lemma~\ref{lem4.2},
\[
c_0<c_{\infty}.
\]
Choose \(\nu>0\) such that
\[
c_0+\nu<\min\{2c_0,c_{\infty}\}.
\]
Since \(c_{\varepsilon}\to c_0\) as \(\varepsilon\to0\), there exists \(\varepsilon_{*}>0\) such that
\[
c_0+\nu<\min\{2c_{\varepsilon},c_{\infty}\}
\qquad \forall\, \varepsilon\in(0,\varepsilon_{*}).
\]

Fix \(\varepsilon\in(0,\varepsilon_{*})\), let
\[
c\in[c_{\varepsilon},c_0+\nu],
\]
and let \(\{u_n\}\subset X_{\varepsilon}\) be a \((PS)_c\)-sequence for \(J_\varepsilon\). By standard arguments, \(\{u_n\}\) is bounded in \(X_\varepsilon\). Up to a subsequence, there exists \(u\in X_\varepsilon\) such that
\[
u_n\rightharpoonup u
\qquad\text{in }X_{\varepsilon},
\]
\[
u_n\to u
\qquad\text{in }L_{\mathrm{loc}}^{r}(\R^{3}),
\qquad \forall\, r\in[1,2_{\alpha}^{*}),
\]
and
\[
u_n(x)\to u(x)
\qquad\text{a.e. in }\R^{3}.
\]
Passing to the limit in \(J_{\varepsilon}'(u_n)\to0\), we obtain
\[
J_{\varepsilon}'(u)=0
\qquad\text{in }X_{\varepsilon}^{*}.
\]

We claim that
\[
u_n\to u
\qquad \text{in }L^{r}(\R^{3}),
\qquad \forall\, r\in[2,2_{\alpha}^{*}).
\]
If not, we may repeat the profile extraction argument used in Lemma~\ref{lem3.7}. Since \(u_n\to u\) locally in \(L^2\), any nontrivial remainder profile must occur at infinity and therefore solves the autonomous limit problem. Consequently one obtains
\[
c\ge J_{\varepsilon}(u)+c_{\infty}.
\]
If \(u=0\), this gives \(c\ge c_\infty\). If \(u\neq0\), then \(u\in\mathcal N_\varepsilon\), and Lemma~\ref{lem3.6}(ii) yields \(J_\varepsilon(u)\ge c_\varepsilon>0\), so again \(c\ge c_\infty\). Both alternatives contradict
\[
c\le c_0+\nu<c_{\infty}.
\]
Thus
\[
u_n\to u
\qquad \text{in }L^{r}(\R^{3}),
\qquad \forall\, r\in[2,2_{\alpha}^{*}).
\]

The proof of Lemma~\ref{lem3.8} now applies and yields
\[
u_n\to u
\qquad \text{strongly in }X_{\varepsilon}.
\]
Hence \(J_\varepsilon\) satisfies the \((PS)_c\) condition for every
\[
c\in[c_{\varepsilon},c_0+\nu].
\]
\end{proof}

\begin{lemma}\label{lem5.3}
There exist \(\nu>0\), \(\sigma>0\), and \(\varepsilon_{*}>0\) such that, for every
\[
\varepsilon\in(0,\varepsilon_{*}),
\qquad
u\in\mathcal N_{\varepsilon},
\qquad
J_{\varepsilon}(u)\le c_0+\nu,
\]
one has
\[
\langle \mathcal H_{\varepsilon}'(u),u\rangle\le -\sigma.
\]
\end{lemma}

\begin{proof}
For \(u\in\mathcal N_{\varepsilon}\), set
\[
h_u(t):=J_{\varepsilon}(tu),
\qquad t\ge0.
\]
Since \(u\in\mathcal N_{\varepsilon}\),
\[
h_u'(1)=\langle J_{\varepsilon}'(u),u\rangle=0.
\]
Using
\[
F_1(s)-F_2(s)=-\frac12 s^2\log s^2,
\]
we may write
\[
\begin{aligned}
h_u(t)
&=
\frac{t^2}{2}\|u\|_{H_{\varepsilon}}^2
+\frac{t^4}{4}\int_{\R^3}\phi_u^\alpha u^2\,dx
-\frac{t^2}{2}\int_{\R^3}u^2\log u^2\,dx \\
&\quad
-\frac{t^2}{2}\log t^2\int_{\R^3}u^2\,dx
-\frac{t^p}{p}\int_{\R^3}|u|^p\,dx.
\end{aligned}
\]
Differentiating twice at \(t=1\), we obtain
\[
\langle \mathcal H_{\varepsilon}'(u),u\rangle
=
h_u''(1)
=
2\int_{\R^3}\phi_u^\alpha u^2\,dx
-2\int_{\R^3}|u|^2\,dx
+(2-p)\int_{\R^3}|u|^p\,dx.
\]

Assume by contradiction that the conclusion fails. Then there exist
\[
\varepsilon_n\to0^+,
\qquad
u_n\in\mathcal N_{\varepsilon_n},
\qquad
J_{\varepsilon_n}(u_n)\le c_0+\frac1n,
\]
such that
\[
\langle \mathcal H_{\varepsilon_n}'(u_n),u_n\rangle\to0.
\]
Since
\[
c_{\varepsilon_n}\le J_{\varepsilon_n}(u_n)\le c_0+\frac1n
\]
and \(c_{\varepsilon_n}\to c_0\), it follows that
\[
J_{\varepsilon_n}(u_n)\to c_0.
\]
By Lemma~\ref{lem4.4}, there exists \(\{\widetilde y_n\}\subset\R^3\) such that
\[
v_n(x):=u_n(x+\widetilde y_n)\to v
\qquad \text{strongly in }X,
\]
where \(v\neq0\) is a ground state solution of the limit problem.

Since
\[
2\int_{\R^3}\phi_u^\alpha u^2\,dx
-2\int_{\R^3}|u|^2\,dx
+(2-p)\int_{\R^3}|u|^p\,dx
\]
is translation invariant,
\[
\langle \mathcal H_{\varepsilon_n}'(u_n),u_n\rangle
=
2\int_{\R^3}\phi_{v_n}^\alpha v_n^2\,dx
-2\int_{\R^3}|v_n|^2\,dx
+(2-p)\int_{\R^3}|v_n|^p\,dx.
\]
Passing to the limit and using Lemma~\ref{lem2.5}(vi), we obtain
\[
\langle \mathcal H_{\varepsilon_n}'(u_n),u_n\rangle
\to
\langle \mathcal H_0'(v),v\rangle.
\]
Since \(v\in\mathcal N_0\), the autonomous analogue of Lemma~\ref{lem3.4} gives
\[
\langle \mathcal H_0'(v),v\rangle<0,
\]
a contradiction.
\end{proof}

\begin{lemma}\label{lem5.4}
Assume that there exist \(d>c_{\varepsilon}\) and \(\sigma>0\) such that
\[
\langle \mathcal H_{\varepsilon}'(u),u\rangle\le -\sigma
\]
for every \(u\in\mathcal N_{\varepsilon}\) satisfying
\[
J_{\varepsilon}(u)\le d.
\]
Assume moreover that \(J_{\varepsilon}\) satisfies the \((PS)_{c}\) condition for every
\[
c\in[c_{\varepsilon},d].
\]
Then the constrained functional
\[
\left.J_{\varepsilon}\right|_{\mathcal N_{\varepsilon}}
\]
satisfies the \((PS)_{c}\) condition for every
\[
c\in[c_{\varepsilon},d].
\]
\end{lemma}

\begin{proof}
Fix \(c\in[c_{\varepsilon},d]\), and let \(\{u_n\}\subset\mathcal N_{\varepsilon}\) be a \((PS)_c\)-sequence for \(\left.J_\varepsilon\right|_{\mathcal N_\varepsilon}\). Then
\[
J_{\varepsilon}(u_n)\to c,
\]
and there exist \(\lambda_n\in\R\) such that
\[
J_{\varepsilon}'(u_n)=\lambda_n \mathcal H_{\varepsilon}'(u_n)+o_n(1)
\qquad\text{in }X_{\varepsilon}^{*}.
\]
Since \(u_n\in\mathcal N_\varepsilon\),
\[
\langle J_{\varepsilon}'(u_n),u_n\rangle=0.
\]
Testing against \(u_n\), we obtain
\[
0=\lambda_n\langle \mathcal H_{\varepsilon}'(u_n),u_n\rangle+o_n(1).
\]
For \(n\) large,
\[
J_{\varepsilon}(u_n)\le d,
\]
hence
\[
\langle \mathcal H_{\varepsilon}'(u_n),u_n\rangle\le -\sigma.
\]
Therefore \(\lambda_n\to0\), and thus
\[
J_{\varepsilon}'(u_n)\to0
\qquad\text{in }X_{\varepsilon}^{*}.
\]
So \(\{u_n\}\) is a \((PS)_c\)-sequence for the ambient functional \(J_\varepsilon\). By the ambient \((PS)_c\) condition, up to a subsequence,
\[
u_n\to u
\qquad\text{in }X_{\varepsilon}.
\]
Hence \(\left.J_\varepsilon\right|_{\mathcal N_\varepsilon}\) satisfies the \((PS)_c\) condition on \([c_\varepsilon,d]\).
\end{proof}

\begin{Cor}\label{cor5.5}
There exist \(\nu>0\) and \(\varepsilon_{*}>0\) such that, for every \(\varepsilon\in(0,\varepsilon_{*})\), the constrained functional
\[
\left.J_{\varepsilon}\right|_{\mathcal N_{\varepsilon}}
\]
satisfies the \((PS)_{c}\) condition for every
\[
c\in[c_{\varepsilon},c_0+\nu].
\]
\end{Cor}

\begin{proof}
Let \(\nu>0\), \(\sigma>0\), and \(\varepsilon_0>0\) be given by Lemma~\ref{lem5.3}. By Lemma~\ref{lem5.2}, after possibly reducing \(\varepsilon_0\), the ambient functional \(J_\varepsilon\) satisfies the \((PS)_c\) condition for every
\[
c\in[c_{\varepsilon},c_0+\nu]
\]
whenever \(\varepsilon\in(0,\varepsilon_0)\). Applying Lemma~\ref{lem5.4} with \(d=c_0+\nu\), we obtain the conclusion.
\end{proof}

\begin{lemma}\label{lem5.6}
There exist \(\nu_{1}>0\) and \(\varepsilon_1>0\) such that, for every \(\varepsilon\in(0,\varepsilon_1)\), every critical point \(u\in \mathcal N_\varepsilon\) of \(J_\varepsilon\) satisfying
\[
J_\varepsilon(u)\le c_0+\nu_{1}
\]
is positive in \(\R^3\).
\end{lemma}

\begin{proof}
Assume by contradiction that there exist \(\varepsilon_n\to0^+\) and critical points
\[
u_n\in \mathcal N_{\varepsilon_n},
\qquad
J_{\varepsilon_n}(u_n)\le c_0+\frac1n,
\]
such that \(u_n\) is not positive in \(\R^3\) for every \(n\).

Since
\[
c_{\varepsilon_n}\le J_{\varepsilon_n}(u_n)\le c_0+\frac1n
\]
and \(c_{\varepsilon_n}\to c_0\), we have
\[
J_{\varepsilon_n}(u_n)\to c_0.
\]
By Lemma~\ref{lem4.4}, there exists a sequence \(\{\widetilde y_n\}\subset\R^3\) such that, setting
\[
v_n(x)=u_n(x+\widetilde y_n),
\qquad
y_n=\varepsilon_n\widetilde y_n,
\]
one has
\[
v_n\to v
\qquad \text{strongly in }X,
\]
where \(v\neq0\) is a ground state solution of the limit problem and
\[
y_n\to y_0\in M.
\]
By Lemma~\ref{lem4.3},
\[
v(x)>0
\qquad \forall\, x\in\R^3.
\]

Each \(v_n\) solves
\[
(-\Delta)^\alpha v_n+\bigl(V(\varepsilon_n x+y_n)+1\bigr)v_n+\phi_{v_n}^{\alpha}v_n
=
F_2'(v_n)-F_1'(v_n)+|v_n|^{p-2}v_n
\qquad \text{in }\R^3.
\]
Set
\[
z_n=|v_n|.
\]
Using the oddness of \(F_1'\) and \(F_2'\), together with the fractional Kato inequality, we obtain
\[
(-\Delta)^\alpha z_n+\bigl(V(\varepsilon_n x+y_n)+1\bigr)z_n+\phi_{v_n}^{\alpha}z_n
\le
F_2'(z_n)-F_1'(z_n)+z_n^{p-1}
\qquad \text{in }\R^3.
\]
Since \(F_1'(t)\ge0\) for \(t\ge0\), \(V(\varepsilon_n x+y_n)+1\ge V_0+1>0\), and \(\phi_{v_n}^\alpha\ge0\), it follows that
\[
(-\Delta)^\alpha z_n
\le
F_2'(z_n)+z_n^{p-1}
\qquad \text{in }\R^3.
\]
By \((f_2)\), there exists \(C>0\) such that
\[
F_2'(t)\le C\bigl(t+t^{p-1}\bigr),
\qquad \forall\, t\ge0,
\]
hence
\[
(-\Delta)^\alpha z_n
\le
C\bigl(z_n+z_n^{p-1}\bigr)
\qquad \text{in }\R^3.
\]
Since \(p<2_\alpha^*\) and \(\{z_n\}\) is bounded in \(L^r(\R^3)\) for every \(r\in[2,2_\alpha^*]\), the standard fractional Moser iteration yields
\[
\|z_n\|_{L^\infty(\R^3)}\le C
\qquad \forall\, n.
\]
Equivalently,
\[
\|v_n\|_{L^\infty(\R^3)}\le C
\qquad \forall\, n.
\]

Set
\[
\mu=\frac{V_0+1}{2}>0,
\qquad
G(t)=\bigl(F_2'(t)-F_1'(t)+t^{p-1}\bigr)^+,
\qquad t\ge0.
\]
Since \(\{z_n\}\) is uniformly bounded in \(L^\infty(\R^3)\), there exists \(M>0\) such that
\[
0\le z_n(x)\le M
\qquad \forall\, x\in\R^3,\ \forall\, n.
\]
Moreover,
\[
F_2'(t)-F_1'(t)+t^{p-1}
=
t\bigl(\log t^2+1+t^{p-2}\bigr)
\qquad \forall\, t>0.
\]
Hence there exists \(\tau\in(0,\delta_0)\) such that
\[
\log t^2+1+t^{p-2}<0
\qquad \forall\, t\in(0,\tau],
\]
and therefore
\[
G(t)=0
\qquad \forall\, t\in[0,\tau].
\]
Since \(G\in C^1([\tau,M])\) and \(G(0)=0\), the function \(G\) is Lipschitz on \([0,M]\). Thus there exists \(L>0\) such that
\[
|G(s)-G(t)|\le L|s-t|,
\qquad
|G(t)|\le Lt,
\qquad \forall\, s,t\in[0,M].
\]

Define
\[
h_n(x)=G(z_n(x)),
\qquad
h(x)=G(v(x)).
\]
Since \(v\ge0\) and
\[
\||v_n|-v\|_{L^2(\R^3)}\le \|v_n-v\|_{L^2(\R^3)}\to0,
\]
the Lipschitz continuity of \(G\) yields
\[
h_n\to h
\qquad \text{in }L^2(\R^3).
\]

From the inequality for \(z_n\) and the choice of \(\mu\), we obtain
\[
(-\Delta)^\alpha z_n+\mu z_n
\le
F_2'(z_n)-F_1'(z_n)+z_n^{p-1}
-
\bigl(V(\varepsilon_n x+y_n)+1-\mu\bigr)z_n
-
\phi_{v_n}^{\alpha}z_n
\le h_n
\]
in \(\R^3\), because
\[
V(\varepsilon_n x+y_n)+1-\mu\ge\mu>0
\qquad\text{and}\qquad
\phi_{v_n}^{\alpha}\ge0.
\]

Let \(K_\mu\) be the Bessel kernel associated with \((-\Delta)^\alpha+\mu\), and define
\[
\zeta_n=K_\mu*h_n,
\qquad
\zeta=K_\mu*h.
\]
Then
\[
(-\Delta)^\alpha \zeta_n+\mu\zeta_n=h_n
\qquad \text{in }\R^3.
\]
Hence
\[
(-\Delta)^\alpha(\zeta_n-z_n)+\mu(\zeta_n-z_n)\ge0
\qquad \text{in }\R^3.
\]
Testing with \((\zeta_n-z_n)^-\), we obtain
\[
0\le z_n\le \zeta_n
\qquad \text{in }\R^3.
\]

Since \(\alpha>\frac34\), we have \(K_\mu\in L^2(\R^3)\). Therefore
\[
\|\zeta_n-\zeta\|_{L^\infty(\R^3)}
\le
\|K_\mu\|_{L^2(\R^3)}\|h_n-h\|_{L^2(\R^3)}
\to0.
\]
Moreover, \(\zeta\in C_0(\R^3)\), and thus
\[
\zeta(x)\to0
\qquad \text{as }|x|\to\infty.
\]
Since \(0\le z_n\le \zeta_n\), it follows that
\[
|v_n(x)|=z_n(x)\to0
\qquad \text{as }|x|\to\infty
\]
uniformly in \(n\).

We next prove local uniform convergence. Since \(v_n\to v\) strongly in \(X\), the sequence \(\{v_n\}\) is bounded in \(H^\alpha(\R^3)\cap L^\infty(\R^3)\). Using the representation formula for \(\phi_{v_n}^\alpha\),
\[
\phi_{v_n}^{\alpha}(x)
=
C_\alpha\int_{\R^3}\frac{v_n(y)^2}{|x-y|^{3-2\alpha}}\,dy,
\]
we obtain a uniform bound
\[
\|\phi_{v_n}^{\alpha}\|_{L^\infty(\R^3)}\le C.
\]
Indeed,
\[
\phi_{v_n}^{\alpha}(x)
\le
C_\alpha\|v_n\|_{L^\infty}^2\int_{B_1(0)}\frac{dz}{|z|^{3-2\alpha}}
+
C_\alpha\int_{|x-y|\ge1}v_n(y)^2\,dy
\le C,
\]
because \(3-2\alpha<3\).

Therefore, on every bounded ball \(B\subset\R^3\), the right-hand side of the equation
\[
(-\Delta)^\alpha v_n
=
-\bigl(V(\varepsilon_n x+y_n)+1\bigr)v_n-\phi_{v_n}^{\alpha}v_n
+
F_2'(v_n)-F_1'(v_n)+|v_n|^{p-2}v_n
\]
is uniformly bounded in \(L^\infty(B)\). Standard regularity for fractional equations then yields a uniform H\"older bound for \(\{v_n\}\) on compact sets, and since \(v_n\to v\) strongly in \(X\), the local compact embedding together with Arzel\`a--Ascoli implies
\[
v_n\to v
\qquad \text{locally uniformly in }\R^3
\]
up to a subsequence.

Choose \(R>0\) so large that
\[
\min_{\overline{B_R(0)}}v>0.
\]
Then, by local uniform convergence,
\[
v_n>0
\qquad \text{in }\overline{B_R(0)}
\]
for all \(n\) sufficiently large. Enlarging \(R\) if necessary, the uniform decay proved above gives
\[
|v_n(x)|\le \tau
\qquad \forall\, |x|\ge R,\ \forall\, n \text{ sufficiently large}.
\]

We claim that \(v_n\ge0\) in \(\R^3\) for all \(n\) large. Otherwise, for infinitely many \(n\), there exists a point \(x_n\in\R^3\) such that
\[
v_n(x_n)=\min_{\R^3}v_n<0.
\]
Since \(v_n>0\) on \(\overline{B_R(0)}\), one has \(|x_n|>R\), and thus
\[
0<|v_n(x_n)|\le \tau.
\]
By the regularity already obtained, the equation is satisfied pointwise. Since \(x_n\) is a global minimum point,
\[
(-\Delta)^\alpha v_n(x_n)\le0.
\]
Evaluating the equation at \(x_n\), we obtain
\[
(-\Delta)^\alpha v_n(x_n)
+
\bigl(V(\varepsilon_n x_n+y_n)+1\bigr)v_n(x_n)
+
\phi_{v_n}^{\alpha}(x_n)v_n(x_n)
=
F_2'(v_n(x_n))-F_1'(v_n(x_n))+|v_n(x_n)|^{p-2}v_n(x_n).
\]
The left-hand side is strictly negative, because
\[
v_n(x_n)<0,
\qquad
V(\varepsilon_n x_n+y_n)+1>0,
\qquad
\phi_{v_n}^{\alpha}(x_n)\ge0.
\]
On the other hand,
\[
F_2'(v_n(x_n))-F_1'(v_n(x_n))+|v_n(x_n)|^{p-2}v_n(x_n)
=
v_n(x_n)\bigl(\log v_n(x_n)^2+1+|v_n(x_n)|^{p-2}\bigr),
\]
and the choice of \(\tau\) implies
\[
\log v_n(x_n)^2+1+|v_n(x_n)|^{p-2}<0.
\]
Hence the right-hand side is strictly positive, a contradiction. Therefore
\[
v_n\ge0
\qquad \text{in }\R^3
\]
for all \(n\) sufficiently large.

Finally, since \(v_n\not\equiv0\) and \(v_n\) is a nonnegative weak solution, the strong maximum principle for the fractional Laplacian yields
\[
v_n>0
\qquad \text{in }\R^3
\]
for all \(n\) sufficiently large. Translating back, we conclude that
\[
u_n>0
\qquad \text{in }\R^3,
\]
contrary to the choice of \(u_n\). This contradiction completes the proof.
\end{proof}

\subsection{Lusternik--Schnirelmann framework and barycenter map}

In this subsection we prove the multiplicity of positive solutions for problem~\eqref{eq1.2} by Lusternik--Schnirelmann category theory.

\begin{Def}\label{def5.7}
Let \(Y\) be a closed subset of a topological space \(Z\). The Lusternik--Schnirelmann category of \(Y\) in \(Z\), denoted by \(\operatorname{cat}_{Z}(Y)\), is the least integer \(n\in\N\) such that \(Y\) can be covered by \(n\) closed sets contractible in \(Z\).
\end{Def}

For a Banach space \(W\), a \(C^{1}\)-submanifold \(\mathcal S\subset W\), a functional \(I:W\to\R\), and \(d\in\R\), define
\[
I^{d}:=\{u\in\mathcal S:\ I(u)\le d\}.
\]

The following abstract result will be used later; see \cite[Chapter~5]{MR1400007}.

\begin{theorem}\label{thm5.8}
Let \(W\) be a Banach space, let \(\mathcal S\subset W\) be a \(C^{1}\)-submanifold, and let \(I\in C^{1}(W,\R)\) be such that \(\left.I\right|_{\mathcal S}\) is bounded from below. Assume that \(\left.I\right|_{\mathcal S}\) satisfies the \((PS)_{c}\) condition for every
\[
c\in\bigl[\inf_{\mathcal S}I,d\bigr].
\]
Then \(\left.I\right|_{\mathcal S}\) has at least
\[
\operatorname{cat}_{I^{d}}(I^{d})
\]
critical points in \(I^{d}\).
\end{theorem}

Let \(u_0\) be a positive ground state solution of the limit problem, that is,
\[
u_0\in\mathcal N_0,
\qquad
J_0(u_0)=c_0.
\]
Fix \(\delta>0\) and choose \(\varphi\in C^{\infty}([0,\infty))\) such that
\[
0\le\varphi\le1,
\qquad
\varphi(t)=1 \ \text{for }0\le t\le \frac{\delta}{2},
\qquad
\varphi(t)=0 \ \text{for }t\ge \delta.
\]
For each \(y\in M:=\{x\in\R^3:V(x)=V_0\}\), define
\[
\Psi_{\varepsilon,y}(x)
=
\varphi(|\varepsilon x-y|)\,
u_0\!\left(x-\frac{y}{\varepsilon}\right).
\]
Since \(\varphi\) has compact support, \(\Psi_{\varepsilon,y}\) has compact support for every \(y\in M\).

By Lemma~\ref{lem3.3}, for every \(y\in M\) there exists a unique \(t_{\varepsilon,y}>0\) such that
\[
t_{\varepsilon,y}\Psi_{\varepsilon,y}\in\mathcal N_{\varepsilon}.
\]
We define
\[
\Phi_{\varepsilon}:M\to\mathcal N_{\varepsilon},
\qquad
\Phi_{\varepsilon}(y):=t_{\varepsilon,y}\Psi_{\varepsilon,y}.
\]

Choose \(\rho>0\) such that
\[
M_{\delta}\subset B_{\rho}(0),
\]
and define \(\zeta:\R^{3}\to\R^{3}\) by
\[
\zeta(x)=
\begin{cases}
x, & |x|\le \rho,\\[2mm]
\rho\dfrac{x}{|x|}, & |x|\ge \rho.
\end{cases}
\]
Finally, define the barycenter map \(\beta:\mathcal N_{\varepsilon}\to\R^{3}\) by
\[
\beta(u)
=
\frac{\displaystyle\int_{\R^{3}}\zeta(\varepsilon x)|u(x)|^{p}\,dx}
{\displaystyle\int_{\R^{3}}|u(x)|^{p}\,dx}.
\]

\begin{lemma}\label{lem5.9}
For every fixed \(\varepsilon>0\), the map
\[
\Phi_\varepsilon:M\to\mathcal N_\varepsilon,
\qquad
\Phi_\varepsilon(y)=t_{\varepsilon,y}\Psi_{\varepsilon,y},
\]
is continuous.
\end{lemma}

\begin{proof}
For fixed \(\varepsilon>0\), set
\[
W_\varepsilon(x):=\varphi(|\varepsilon x|)u_0(x),
\qquad x\in\R^3.
\]
Then
\[
\Psi_{\varepsilon,y}(x)=W_\varepsilon\!\left(x-\frac{y}{\varepsilon}\right),
\qquad \forall\, x\in\R^3,\ y\in M.
\]
Since translations are continuous in \(H^\alpha(\R^3)\) and in \(L^{F_1}(\R^3)\), the map \(y\mapsto \Psi_{\varepsilon,y}\) is continuous from \(M\) into \(X_\varepsilon\).

Let \(y_n\to y\) in \(M\). Then \(\Psi_{\varepsilon,y_n}\to \Psi_{\varepsilon,y}\) in \(X_\varepsilon\). Set \(t_n:=t_{\varepsilon,y_n}\). Exactly as in the proof of Lemma~\ref{lem3.5}, the Nehari identity for \(t_n\Psi_{\varepsilon,y_n}\) shows that \(\{t_n\}\) is bounded and bounded away from zero. Passing to a subsequence, we may assume \(t_n\to t_0>0\). Since \(\Psi_{\varepsilon,y_n}\to \Psi_{\varepsilon,y}\) in \(X_\varepsilon\), we may pass to the limit in
\[
\langle J_\varepsilon'(t_n\Psi_{\varepsilon,y_n}),t_n\Psi_{\varepsilon,y_n}\rangle=0
\]
and obtain
\[
\langle J_\varepsilon'(t_0\Psi_{\varepsilon,y}),t_0\Psi_{\varepsilon,y}\rangle=0.
\]
Hence \(t_0\Psi_{\varepsilon,y}\in\mathcal N_\varepsilon\). By the uniqueness of the fibering maximum from Lemma~\ref{lem3.3}, we conclude that \(t_0=t_{\varepsilon,y}\). Therefore the whole sequence satisfies
\[
t_{\varepsilon,y_n}\to t_{\varepsilon,y}.
\]
Therefore
\[
\Phi_\varepsilon(y_n)=t_{\varepsilon,y_n}\Psi_{\varepsilon,y_n}
\to
t_{\varepsilon,y}\Psi_{\varepsilon,y}=\Phi_\varepsilon(y)
\qquad \text{in }X_\varepsilon,
\]
which proves the continuity of \(\Phi_\varepsilon\).
\end{proof}

\begin{lemma}\label{lem5.10}
One has
\[
\lim_{\varepsilon\to0}J_{\varepsilon}(\Phi_{\varepsilon,y})=c_0
\qquad\text{uniformly in }y\in M.
\]
\end{lemma}

\begin{proof}
Assume by contradiction that there exist \(\tau_0>0\), \(\varepsilon_n\to0\), and \(y_n\in M\) such that
\[
\bigl|J_{\varepsilon_n}(\Phi_{\varepsilon_n,y_n})-c_0\bigr|\ge \tau_0
\qquad \forall\, n.
\]
Set
\[
t_n:=t_{\varepsilon_n,y_n},
\qquad
\Psi_n:=\Psi_{\varepsilon_n,y_n}.
\]
With the change of variables
\[
x=z+\frac{y_n}{\varepsilon_n},
\]
we have
\[
\Psi_n\!\left(z+\frac{y_n}{\varepsilon_n}\right)=\varphi(|\varepsilon_n z|)u_0(z).
\]
By dominated convergence, Proposition~\ref{prop2.1}, and Lemma~\ref{lem2.6},
\[
\varphi(|\varepsilon_n \cdot|)u_0 \to u_0
\qquad \text{in }X.
\]
Since \(M\) is compact and \(V\) is continuous,
\[
\|\Psi_n\|_{H_{\varepsilon_n}}^{2}\to\|u_0\|_{H_0}^{2},
\qquad
\int_{\R^{3}}\phi_{\Psi_n}^{\alpha}\Psi_n^{2}\,dx
\to
\int_{\R^{3}}\phi_{u_0}^{\alpha}u_0^{2}\,dx,
\]
\[
\int_{\R^{3}}|\Psi_n|^{p}\,dx\to\int_{\R^{3}}|u_0|^{p}\,dx,
\qquad
\int_{\R^{3}}\bigl(F_1(\Psi_n)-F_2(\Psi_n)\bigr)\,dx
\to
\int_{\R^{3}}\bigl(F_1(u_0)-F_2(u_0)\bigr)\,dx.
\]
Since \(t_n\Psi_n\in\mathcal N_{\varepsilon_n}\), the sequence \(\{t_n\}\) is bounded and bounded away from \(0\). Thus, up to a subsequence,
\[
t_n\to t_0>0.
\]
Passing to the limit in the Nehari identity, we obtain
\[
\langle J_0'(t_0u_0),t_0u_0\rangle=0.
\]
Hence \(t_0u_0\in\mathcal N_0\). Since also \(u_0\in\mathcal N_0\), Lemma~\ref{lem3.3} gives \(t_0=1\). Therefore \(t_n\to1\), and
\[
J_{\varepsilon_n}(\Phi_{\varepsilon_n,y_n})
=
J_{\varepsilon_n}(t_n\Psi_n)\to J_0(u_0)=c_0,
\]
a contradiction.
\end{proof}

Define
\[
\alpha(\varepsilon):=\sup_{y\in M}\bigl|J_{\varepsilon}(\Phi_{\varepsilon,y})-c_0\bigr|.
\]
Then
\[
\alpha(\varepsilon)\to0
\qquad \text{as }\varepsilon\to0.
\]
Set
\[
\widetilde{\mathcal N}_{\varepsilon}
=
\bigl\{
u\in\mathcal N_{\varepsilon}:\ J_{\varepsilon}(u)\le c_0+\alpha(\varepsilon)
\bigr\}.
\]
Then \(\Phi_{\varepsilon}(M)\subset\widetilde{\mathcal N}_{\varepsilon}\) for \(\varepsilon>0\) sufficiently small.

\begin{lemma}\label{lem5.11}
The barycenter map satisfies
\[
\lim_{\varepsilon\to0}\beta\bigl(\Phi_{\varepsilon,y}\bigr)=y
\qquad \text{uniformly for } y\in M.
\]
\end{lemma}

\begin{proof}
Fix sequences \(\varepsilon_n\to0\) and \(y_n\in M\). By the change of variables \(x=z+y_n/\varepsilon_n\),
\[
\beta\bigl(\Phi_{\varepsilon_n,y_n}\bigr)
=
\frac{\displaystyle\int_{\R^{3}}\zeta(\varepsilon_n z+y_n)\bigl|\varphi(|\varepsilon_n z|)u_0(z)\bigr|^p\,dz}
{\displaystyle\int_{\R^{3}}\bigl|\varphi(|\varepsilon_n z|)u_0(z)\bigr|^p\,dz}.
\]
Since \(y_n\in M\subset B_{\rho}(0)\), one has \(\zeta(y_n)=y_n\), and therefore
\[
\beta\bigl(\Phi_{\varepsilon_n,y_n}\bigr)-y_n
=
\frac{\displaystyle\int_{\R^{3}}\bigl(\zeta(\varepsilon_n z+y_n)-y_n\bigr)\bigl|\varphi(|\varepsilon_n z|)u_0(z)\bigr|^p\,dz}
{\displaystyle\int_{\R^{3}}\bigl|\varphi(|\varepsilon_n z|)u_0(z)\bigr|^p\,dz}.
\]
The denominator converges to \(\int_{\R^{3}}|u_0|^p\,dz>0\). Since \(\zeta\) is \(1\)-Lipschitz,
\[
|\zeta(\varepsilon_n z+y_n)-y_n|
=
|\zeta(\varepsilon_n z+y_n)-\zeta(y_n)|
\le |\varepsilon_n z|.
\]
On the support of \(\varphi(|\varepsilon_n z|)\), one has \(|\varepsilon_n z|\le\delta\), and dominated convergence yields the conclusion.
\end{proof}

\begin{lemma}\label{lem5.12}
There exists \(\varepsilon_1>0\) such that, for every \(\varepsilon\in(0,\varepsilon_1)\),
\[
\beta(u)\in M_{\delta},
\qquad \forall\, u\in\widetilde{\mathcal N}_{\varepsilon}.
\]
\end{lemma}

\begin{proof}
Assume by contradiction that there exist \(\varepsilon_n\to0\) and \(u_n\in\widetilde{\mathcal N}_{\varepsilon_n}\) such that
\[
\beta(u_n)\notin M_{\delta}
\qquad \forall\, n.
\]
Since \(u_n\in\widetilde{\mathcal N}_{\varepsilon_n}\),
\[
c_{\varepsilon_n}\le J_{\varepsilon_n}(u_n)\le c_0+\alpha(\varepsilon_n).
\]
By Lemma~\ref{lem3.6}(iii), \(c_{\varepsilon_n}\to c_0\), so
\[
J_{\varepsilon_n}(u_n)\to c_0.
\]
Applying Lemma~\ref{lem4.4}, there exists \(\{\widetilde y_n\}\subset\R^3\) such that
\[
v_n(x):=u_n(x+\widetilde y_n)\to v
\qquad \text{strongly in }X,
\]
and
\[
y_n:=\varepsilon_n\widetilde y_n\to y_0\in M.
\]
Moreover,
\[
\beta(u_n)
=
\frac{\displaystyle\int_{\R^{3}}\zeta(\varepsilon_n z+y_n)|v_n(z)|^p\,dz}
{\displaystyle\int_{\R^{3}}|v_n(z)|^p\,dz}.
\]
By the strong convergence in \(L^p(\R^3)\),
\[
\int_{\R^{3}}|v_n|^p\,dz\to\int_{\R^{3}}|v|^p\,dz>0.
\]
Since \(\zeta\) is bounded, \(\zeta(\varepsilon_n z+y_n)\to y_0\) a.e., and \(v_n\to v\) strongly in \(L^p(\R^3)\), we obtain
\[
\beta(u_n)\to y_0\in M,
\]
contradicting \(\beta(u_n)\notin M_{\delta}\).
\end{proof}

\subsection{Proof of Theorem~\ref{thm1.1}}

\begin{lemma}\label{lem5.13}
There exists \(\varepsilon_0>0\) such that, for every \(\varepsilon\in(0,\varepsilon_0)\), the rescaled problem associated with \(J_{\varepsilon}\) has at least
\[
\operatorname{cat}_{M_{\delta}}(M)
\]
distinct positive solutions.
\end{lemma}

\begin{proof}
Let \(\nu>0\) and \(\varepsilon_{\mathrm{PS}}>0\) be given by Corollary~\ref{cor5.5}. Let \(\nu_1>0\) and \(\varepsilon_{\mathrm{pos}}>0\) be given by Lemma~\ref{lem5.6}, and let \(\varepsilon_{\mathrm{bar}}>0\) be given by Lemma~\ref{lem5.12}. Since
\[
\alpha(\varepsilon)\to0,
\]
after possibly reducing \(\varepsilon_{\mathrm{PS}}\), we may assume that
\[
\alpha(\varepsilon)<\min\{\nu,\nu_1\}
\qquad \forall\, \varepsilon\in(0,\varepsilon_{\mathrm{PS}}).
\]
Fix
\[
\varepsilon\in\bigl(0,\min\{\varepsilon_{\mathrm{PS}},\varepsilon_{\mathrm{bar}},\varepsilon_{\mathrm{pos}}\}\bigr).
\]
Then Corollary~\ref{cor5.5} implies that \(\left.J_{\varepsilon}\right|_{\mathcal N_{\varepsilon}}\) satisfies the \((PS)_c\) condition for every
\[
c\in[c_{\varepsilon},\,c_0+\alpha(\varepsilon)].
\]
Moreover, by Lemma~\ref{lem3.6}(ii),
\[
\inf_{\mathcal N_{\varepsilon}}J_{\varepsilon}=c_{\varepsilon}.
\]
Applying Theorem~\ref{thm5.8} with
\[
W=X_{\varepsilon},
\qquad
\mathcal S=\mathcal N_{\varepsilon},
\qquad
I=J_{\varepsilon},
\qquad
d=c_0+\alpha(\varepsilon),
\]
we obtain at least
\[
\operatorname{cat}_{\widetilde{\mathcal N}_{\varepsilon}}(\widetilde{\mathcal N}_{\varepsilon})
\]
critical points of \(\left.J_{\varepsilon}\right|_{\mathcal N_{\varepsilon}}\) in \(\widetilde{\mathcal N}_{\varepsilon}\). By Lemma~\ref{lem5.1}, these are free critical points of \(J_{\varepsilon}\), and by Lemma~\ref{lem5.6} they are positive solutions of the rescaled problem.

It remains to prove that
\[
\operatorname{cat}_{\widetilde{\mathcal N}_{\varepsilon}}(\widetilde{\mathcal N}_{\varepsilon})
\ge
\operatorname{cat}_{M_{\delta}}(M).
\]
Set
\[
m:=\operatorname{cat}_{\widetilde{\mathcal N}_{\varepsilon}}(\widetilde{\mathcal N}_{\varepsilon}),
\]
and let \(A_1,\dots,A_m\subset\widetilde{\mathcal N}_{\varepsilon}\) be closed sets, each contractible in \(\widetilde{\mathcal N}_{\varepsilon}\), such that
\[
\widetilde{\mathcal N}_{\varepsilon}=\bigcup_{i=1}^{m}A_i.
\]
Define
\[
B_i:=\Phi_{\varepsilon}^{-1}(A_i).
\]
Since \(\Phi_{\varepsilon}(M)\subset\widetilde{\mathcal N}_{\varepsilon}\),
\[
M=\bigcup_{i=1}^{m}B_i,
\]
and each \(B_i\) is closed in \(M\) by Lemma~\ref{lem5.9}.

By Lemma~\ref{lem5.11}, the map
\[
\beta\circ\Phi_{\varepsilon}:M\to M_{\delta}
\]
is homotopic in \(M_{\delta}\) to the inclusion \(i:M\hookrightarrow M_{\delta}\). Indeed, defining
\[
\eta(t,y)=(1-t)y+t\,\beta(\Phi_{\varepsilon}(y)),
\qquad (t,y)\in[0,1]\times M,
\]
we have
\[
\operatorname{dist}(\eta(t,y),M)
\le
|\eta(t,y)-y|
=
t\,|\beta(\Phi_{\varepsilon}(y))-y|.
\]
By Lemma~\ref{lem5.11}, for \(\varepsilon\) sufficiently small,
\[
|\beta(\Phi_{\varepsilon}(y))-y|<\delta
\qquad \forall\, y\in M,
\]
and hence
\[
\eta([0,1]\times M)\subset M_{\delta}.
\]

For each \(i\), let
\[
h_i:[0,1]\times A_i\to\widetilde{\mathcal N}_{\varepsilon}
\]
be a contraction of \(A_i\) in \(\widetilde{\mathcal N}_{\varepsilon}\), so that
\[
h_i(0,u)=u,
\qquad
h_i(1,u)=u_i^{*}
\]
for some fixed \(u_i^{*}\in\widetilde{\mathcal N}_{\varepsilon}\). Define
\[
H_i:[0,1]\times B_i\to M_{\delta}
\]
by
\[
H_i(t,y)=
\begin{cases}
\eta(2t,y), & 0\le t\le \frac12,\\[2mm]
\beta\bigl(h_i(2t-1,\Phi_{\varepsilon}(y))\bigr), & \frac12\le t\le1.
\end{cases}
\]
By Lemma~\ref{lem5.12}, this map is well defined, continuous, and satisfies
\[
H_i(0,y)=y,
\qquad
H_i(1,y)=\beta(u_i^{*}),
\]
where the endpoint is independent of \(y\in B_i\). Thus each \(B_i\) is contractible in \(M_{\delta}\). Therefore
\[
\operatorname{cat}_{M_{\delta}}(M)\le m
=
\operatorname{cat}_{\widetilde{\mathcal N}_{\varepsilon}}(\widetilde{\mathcal N}_{\varepsilon}).
\]
The proof is complete.
\end{proof}

\begin{proof}[Proof of Theorem~\ref{thm1.1}]
The existence statement follows from Lemma~\ref{lem5.13} after passing from the rescaled problem to the original one.

Let
\[
(v_{\varepsilon},\psi_{\varepsilon})
\]
be one of the positive solutions of the rescaled system
\[
\begin{cases}
(-\Delta)^\alpha v+V(\varepsilon x)v+\psi v=v\log v^2+|v|^{p-2}v, & \text{in }\R^3,\\
(-\Delta)^\alpha \psi=v^2, & \text{in }\R^3.
\end{cases}
\]
Define
\[
u_{\varepsilon}(x):=v_{\varepsilon}\!\left(\frac{x}{\varepsilon}\right),
\qquad
\phi_{\varepsilon}(x):=\psi_{\varepsilon}\!\left(\frac{x}{\varepsilon}\right).
\]
By the scaling property of the fractional Laplacian, \((u_{\varepsilon},\phi_{\varepsilon})\) solves the original problem \eqref{eq1.1}. Since the change of variables is invertible, distinct positive solutions of the rescaled problem yield distinct positive solutions of \eqref{eq1.1}. Hence \eqref{eq1.1} has at least
\[
\operatorname{cat}_{M_{\delta}}(M)
\]
distinct positive solutions.

It remains to prove the concentration property. Let \(\varepsilon_n\to0^+\), and let
\[
(v_n,\psi_n)
\]
be a sequence of positive solutions of the rescaled problem given by Lemma~\ref{lem5.13}, with
\[
v_n\in \widetilde{\mathcal N}_{\varepsilon_n}.
\]
Since
\[
c_{\varepsilon_n}\le J_{\varepsilon_n}(v_n)\le c_0+\alpha(\varepsilon_n),
\]
and
\[
c_{\varepsilon_n}\to c_0,
\qquad
\alpha(\varepsilon_n)\to0,
\]
we have
\[
J_{\varepsilon_n}(v_n)\to c_0.
\]
By Lemma~\ref{lem4.4}, there exists \(\{\bar y_n\}\subset\R^3\) such that, setting
\[
w_n(x):=v_n(x+\bar y_n),
\qquad
y_n:=\varepsilon_n\bar y_n,
\]
we have, up to a subsequence,
\[
w_n\to \bar u
\qquad \text{strongly in }X,
\]
where \(\bar u\neq0\) is a ground state solution of the limit problem, and
\[
y_n\to y_0\in M.
\]

By Lemma~\ref{lem4.5},
\[
\|w_n\|_{L^\infty(\R^3)}\le C
\qquad \forall\, n,
\]
and
\[
w_n(x)\to0
\qquad \text{as }|x|\to\infty
\]
uniformly in \(n\). Since \(\bar u>0\) by Lemma~\ref{lem4.3}, there exists \(x_0\in\R^3\) such that
\[
\bar u(x_0)>0.
\]
Using the equation for \(w_n\), the uniform \(L^\infty\)-bound, and \cite[Theorem~5.4]{MR2244602}, we obtain uniform H\"older bounds on bounded sets. Since \(w_n\to\bar u\) strongly in \(X\), the local compact embedding together with Arzel\`a--Ascoli implies
\[
w_n\to\bar u
\qquad \text{locally uniformly in }\R^3
\]
up to a subsequence. Hence
\[
w_n(x_0)\to \bar u(x_0)>0.
\]
Therefore there exists \(c>0\) such that
\[
\|w_n\|_{L^\infty(\R^3)}\ge w_n(x_0)\ge c
\qquad \text{for all sufficiently large } n.
\]

Let \(\widehat y_n\in\R^3\) be a global maximum point of \(w_n\). Since \(w_n(x)\to0\) uniformly as \(|x|\to\infty\), there exists \(R>0\) such that
\[
|w_n(x)|<\frac c2
\qquad \forall\, |x|\ge R,\ \forall\, n.
\]
Thus
\[
|\widehat y_n|<R
\qquad \forall\, n,
\]
and, up to a subsequence,
\[
\widehat y_n\to \widehat y
\qquad \text{for some } \widehat y\in\R^3.
\]

Set
\[
x_n:=\widehat y_n+\bar y_n.
\]
Since
\[
w_n(x)=v_n(x+\bar y_n),
\]
the point \(x_n\) is a global maximum point of \(v_n\). Define
\[
u_n(x):=v_n\!\left(\frac{x}{\varepsilon_n}\right),
\qquad
\phi_n(x):=\psi_n\!\left(\frac{x}{\varepsilon_n}\right),
\qquad
\eta_n:=\varepsilon_n x_n.
\]
Then \(\eta_n\) is a global maximum point of \(u_n\), and
\[
\eta_n
=
\varepsilon_n\widehat y_n+\varepsilon_n\bar y_n
=
\varepsilon_n\widehat y_n+y_n.
\]
Since \(\{\widehat y_n\}\) is bounded and \(\varepsilon_n\to0\),
\[
\varepsilon_n\widehat y_n\to0.
\]
Therefore
\[
\eta_n\to y_0\in M.
\]
By continuity of \(V\),
\[
V(\eta_n)\to V_0.
\]
This proves the concentration property and completes the proof of Theorem~\ref{thm1.1}.
\end{proof}

\section*{Acknowledgments}

\medskip
{\bf Funding:} This work is supported by National Natural Science Foundation of China (12301145, 12261107, 12561020) and Yunnan Fundamental Research Projects (202301AU070144, 202401AU070123).

\medskip
{\bf Author Contributions:} All the authors wrote the main manuscript text together and these authors contributed equally to this work.

\medskip
{\bf Data availability:}  Data sharing is not applicable to this article as no new data were created or analyzed in this study.

\medskip
{\bf Conflict of Interests:} The authors declare that there is no conflict of interest.

\bibliographystyle{plain}
\bibliography{reference}

\end{document}